\documentclass[a4paper,12pt,reqno]{amsart}

\usepackage{extsizes}
\usepackage{amsmath,amsthm,amssymb}
\usepackage{comment}
\usepackage[top=2cm,bottom=2cm,left=3cm,right=3cm]{geometry}
\usepackage{txfonts}

\makeatletter

\@addtoreset{equation}{section}
\makeatother

\title{On modular linear differential operators and their applications}
\author{Fumitoshi Yamashita}
\address{Graduate School of Mathematical Sciences, The University of Tokyo, 3-8-1 Komaba Meguro-ku Tokyo 153-8914, Japan}
\email{yamast@ms.u-tokyo.ac.jp}

\date{}


\newtheoremstyle{mytheorem}
{}
{}
{\slshape}
{}
{\bfseries}
{.}
{4pt}
{\thmname{#1}\ \thmnumber{#2}\thmnote{\hspace{2pt}(#3)}}
\theoremstyle{mytheorem}

\newtheorem{thm}{Theorem}[section]
\newtheorem{cor}[thm]{Corollary}
\newtheorem{prop}[thm]{Proposition}
\newtheorem{lem}[thm]{Lemma}
\newtheorem{conj}[thm]{Conjecture}

\theoremstyle{definition}

\newtheorem{exmp}[thm]{Example}

\theoremstyle{remark}
\newtheorem{rem}[thm]{Remark}

\newcommand{\sltwoz}{\mathrm{SL}(2,\mathbb{Z})}

\newcommand{\calh}{\mathcal{H}}
\newcommand{\calm}{\mathcal{M}}
\newcommand{\cals}{\mathcal{S}}

\newcommand{\calr}{\mathcal{R}}
\newcommand{\calqm}{\mathcal{QM}}
\newcommand{\calqr}{\mathcal{QR}}
\newcommand{\calqreone}{\mathcal{QR}\langle e_1 \rangle}
\newcommand{\calqrftwo}{\mathcal{QR}\langle f_2 \rangle}
\newcommand{\mcalr}{{}^{\mathrm{M}}\calr}
\newcommand{\acalr}{{}^{\mathrm{QM}}\calr}
\newcommand{\ocalm}{\overline{\mathcal{M}}}
\newcommand{\ocalr}{\overline{\mathcal{R}}}
\newcommand{\overs}{\overline{S}}
\newcommand{\osigma}{\overline{\Sigma}}

\newcommand{\zz}{\mathbb{Z}}
\newcommand{\qq}{\mathbb{Q}}
\newcommand{\rr}{\mathbb{R}}
\newcommand{\cc}{\mathbb{C}}

\newcommand{\hol}{\mathrm{Hol}}
\newcommand{\mer}{\mathrm{Mer}}
\newcommand{\triv}{\mathrm{triv}}
\newcommand{\ord}{\mathrm{ord}}
\newcommand{\End}{\mathrm{End}}
\newcommand{\topp}{\mathrm{top}}
\newcommand{\wt}{\mathrm{wt}}
\newcommand{\dwt}{\mathrm{dwt}}
\newcommand{\ch}{\mathrm{ch}}
\newcommand{\rmt}{\mathrm{T}}
\newcommand{\mord}{\mathrm{mord}}
\newcommand{\rmd}{\mathrm{D}}

\newcommand{\ann}{\mathrm{ann}}
\newcommand{\gcrd}{\mathrm{gcrd}}
\newcommand{\lclm}{\mathrm{lclm}}

\newcommand{\srdiff}[2]{D_{#2}^{(#1)}}
\newcommand{\dd}{\delta}

\begin{document}

\maketitle

\begin{abstract}
A formal definition of the graded algebra $\calr$ of modular linear differential operators is given and its properties are studied.
An algebraic structure of the solutions to modular linear differential equations (MLDEs) is shown.
It is also proved that any quasimodular form of weight $k$ and depth $s$ becomes a solution to a monic MLDE of weight $k-s$.
By using the algebraic properties of $\calr$, linear differential operators which map the solution space of a monic MLDE to that of another are determined for sufficiently low weights and orders.
Furthermore, a lower bound of the order of monic MLDEs satisfied by ${E_4}^m{E_6}^n$ is found.
\end{abstract}

\section{Introduction}
A modular linear differential equation (MLDE) is a linear ordinary differential equation written in terms of the Ramanujan-Serre differential operators and modular forms on $\sltwoz$.
One of the characteristic properties of MLDE is that the solution space of an MLDE is invariant under the modular transformations of $\sltwoz$.
The most well-known example of MLDEs is the Kaneko-Zagier equation, which is derived from the study of elliptic curves (\cite{kaneko_zagier1998}).
The solutions and the solution spaces of the Kaneko-Zagier equation are studied in many papers (\cite{kaneko_koike2003}, \cite{kaneko_nagatomo_sakai2017} etc.).

Since an MLDE is a linear differential equation, the solution space is the kernel of the differential operator.
Therefore, its properties are important for studying MLDEs.
Although some papers mention such differential operators, they have not been the subject of research and the algebraic properties remain unknown.

In this paper,  we introduce modular linear differential operators (MLDOs).
They correspond to MLDE's differential operators and MLDEs can be expressed by using  MLDOs.
They constitute a graded $\cc$-algebra, which we denote by $\calr$.
We then investigate algebraic properties of MLDOs.
In particular, we determine a $\cc$-basis of $\calr$ (Theorem \ref{thm201712301243} and \ref{thm201712301244}), prove that $\calr$ is isomorphic to a graded skew polynomial ring (Theorem \ref{thm201805311336}) and establish a division with remainder for MLDOs (Theorems \ref{thm201712301245} and \ref{thm201712301246}).
The division properties are used in the subsequent sections many times.

Utilizing the results mentioned above, we show that the solutions to MLDEs have some algebraic structure (Theorems \ref{thm201804191022}, \ref{thm201804240953} and \ref{thm20180624}).
We also prove that a quasimodular form of weight $k$ and depth $s$ satisfies a monic MLDE of weight $k-s$ (Theorem \ref{thm201804241015}).

We next study linear differential operators acting on the solution spaces of monic MLDEs.
The algebraic properties of MLDOs enable us to give a condition under which an MLDO maps the solution space of a monic MLDE to that of another (Corollary \ref{cor201801191507}).
Note that the idea of such a linear differential operator is utilized in the study of lower order monic MLDEs (see, for example, Proposition 1 and the following Lemma in \cite{kaneko_koike2003}) although it has not been systematically studied.

As applications of Corollary \ref{cor201801191507}, we introduce a family of third order monic MLDEs $\phi_pf=0$  (Proposition \ref{prop201801151023}), where $p\in\rr$ is a parameter and $f$ is the unknown function.
The solutions to the MLDE $\phi_pf=0$ are related to the theta functions of the $\rmd_n$ lattices ($p=2n$, Proposition \ref{prop201801151033}).
Then we determine the linear differential operators of low order and weight which map the solution space of the MLDE $\phi_pf=0$ to the solution space of a third order monic MLDE (Example \ref{ex201807171044}).

Moreover, we apply the algebraic properties of MLDOs to the monic MLDEs satisfied by the modular forms ${E_4}^m{E_6}^n$, where $E_k$ is the Eisenstein series of weight $k$.
Although it is known that ${E_4}^m{E_6}^n$ satisfies some monic MLDE, the order of such a monic MLDE is not clear.
Using the division properties of MLDOs, we give a lower bound of the order of monic MLDEs satisfied by ${E_4}^m{E_6}^n$ (Theorem \ref{thm201801111417}).

The paper is organized as follows.
In Section \ref{201712301223}, we review the theory of modular forms and MLDEs.
At the end of the section, we slightly generalize the setting of Mason's theorem (Theorem \ref{thm201712301224}).
In Section \ref{section201712301238}, we give the formal definition of MLDOs, and then study their algebraic properties.
In Section \ref{section20180614}, we show an algebraic structure of the solutions to MLDEs and prove that every quasimodular form satisfies a monic MLDE.
In Section \ref{section201712301252}, we apply the results of Section \ref{section201712301238} to the linear differential operators which map the solution space of a monic MLDE to that of another.
In Section \ref{section201712301302}, we give a lower bound of the order of monic MLDEs satisfied by ${E_4}^m{E_6}^n$.
In Section \ref{section20180220}, we show some properties of left ideals of $\calr$.
We also give $\cc$-algebras containing $\calr$ and determine their endomorphisms.

\section{Preliminaries}
\label{201712301223}

\subsection{Modular forms}
Hereafter, the term modular form will refer to a vector-valued modular form. 

Recall that the \textit{modular group} $\sltwoz=\{A\in\mathrm{M}(2,\zz)\mid\det(A)=1\}$ is generated by $S=(\begin{smallmatrix}0&-1\\1&0\end{smallmatrix})$ and $T=(\begin{smallmatrix}1&1\\0&1\end{smallmatrix})$.

We will denote by $\hol$ (resp.\ $\mer$) the space of all holomorphic (resp.\ meromorphic) functions on the upper half plane $\calh=\{z\in\cc\mid\Im(z)>0\}$, where $\Im(z)$ is the imaginary part of $z$.

For $z\in\calh$ and $\alpha\in\cc$, we set $e(\alpha)=e^{2\pi i\alpha}$, $q=e^{2\pi i z}$, $q^{\alpha}=e^{2\pi i\alpha z}$ and $\log q=2\pi iz$. 
In this paper, we use $z$ (resp.\ $w$) as a variable in $\calh$ (resp.\ $\cc$).
For $w\neq 0$, $\arg(w)\in(-\pi,\pi]$ is the argument of $w$.
Unless otherwise specified, the logarithm and the power of $w$ are calculated by using $\arg(w)$.

For $A=(\begin{smallmatrix}a&b\\c&d\end{smallmatrix})\in\sltwoz$, we set $Az=\frac{az+b}{cz+d}$ and $j(A,z)=cz+d$.
For $f\in\hol$ and $k\in\rr$, we set $(f|_kA)(z)=j(A,z)^{-k}f(Az)$ and for $F=(f_1,\ldots,f_n)^{\rmt}\in\hol^n$, $F|_kA=(f_1|_k A,\ldots,f_n|_k A)^{\rmt}$.

We say that $f$ \textit{has at most exponential growth around infinity} if there exists $C>0$ and $M>0$ such that $|f(z)|<e^{C\Im(z)}$ for all $z$ with $\Im(z)>M$.
We say that $f$ \textit{is bounded around infinity} if there exists $\epsilon>0$ and $M>0$ such that $|f(z)|<\epsilon$ for all $z$ with $\Im(z)>M$.
We say that $f$ \textit{vanishes around infinity} if for all $\epsilon>0$ there exists an $M>0$ such that $|f(z)|<\epsilon$ for all $z$ with $\Im(z)>M$.

Let $\Gamma$ be a subgroup of $\sltwoz$, $v$ a multiplier of weight $k\in\rr$ on $\Gamma$, and $R$ a representation of $\Gamma$ of dimension $n\in\zz_{>0}$.
Then $F\in\hol^n$ is called a \textit{weakly holomorphic modular form} (resp.\ \textit{holomorphic modular form}, resp.\ \textit{cusp form}) \textit{of weight $k$ on $\Gamma$ with respect to $v$ and $R$} if (1) for all $B\in\Gamma$, $F|_kB=v(B)R(B)F$ and (2) for all $A\in\sltwoz$, each component of $F|_k A$ has at most exponential growth (resp.\ is bounded, resp.\ vanishes) around infinity.
We will denote the spaces of weakly holomorphic modular forms, holomorphic modular forms and cusp forms by $\calm^!(k,\Gamma,v,R)$, $\calm(k,\Gamma,v,R)$ and $\cals(k,\Gamma,v,R)$, respectively.
Additionally, we set $\mathcal{X}(k,v)=\mathcal{X}(k,\sltwoz,v,\triv)$ and $\mathcal{X}(k)=\mathcal{X}_k=\mathcal{X}(k,1)$ for $\mathcal{X}=\calm$, $\cals$ or $\calm^!$, where $\triv$ is the trivial one-dimensional representation of $\sltwoz$.

\bigskip
Let $f(z)=q^{\lambda}\sum_{n=0}^{\infty}a_n q^{mn}$ be a convergent $q$-series with $\lambda\in\cc$, $a_n\in\cc$, $a_0\neq0$ and $m\in\rr_{>0}$.
Then $\lambda$ is called the \textit{leading exponent} and $1/m$ is called the \textit{width}.
Suppose $f(z)\neq 0$ for $\Im(z)>N$.
Then for $p\in\rr$, the $p$-th power of $f(z)$ is defined by $f^p(z)=q^{p\lambda}{a_0}^p(1+\binom{p}{1}\frac{a_1}{a_0}q^m+\cdots)$ for $\Im(z)>N$.
Note that $|f^p(z)|=|f(z)|^p$ and $f^{p+p'}(z)=f^p(z)f^{p'}(z)$.

For an even integer $k\ge 2$, the weight $k$ \textit{Eisenstein series} $E_k(z)=1+O(q)$ is defined as usual.
For $k\ge 4$, $E_k\in\calm(k)$ holds, but $E_2\notin \calm(2)$ and we have
\begin{equation}
\label{eq201807171108}
(E_2|_2 A)(z)=E_2(z)+\frac{12c}{2\pi i(cz+d)}
\end{equation} 
for $A=(\begin{smallmatrix}a&b\\c&d\end{smallmatrix})\in\sltwoz$.
The \textit{Dedekind eta function} $\eta(z)=q^{1/24}+O(q^{25/24})\in\cals(1/2,\chi)$ never vanishes on $\calh$, where $\chi$ is a multiplier of weight $1/2$.
The \textit{Ramanujan delta function} is  $\Delta=\frac{1}{1728}({E_4}^3-{E_6}^2)=\eta^{24}\in\cals(12)$.

In this paper, we denote $f'(z)=\frac{1}{2\pi i}\frac{df}{dz}(z)$.
The following identities are called the \textit{Ramanujan identities}: $E_2'=\frac{1}{12}(E_2^2-E_4)$, $E_4'=\frac{1}{3}(E_2E_4-E_6)$ and $E_6'=\frac{1}{2}(E_2E_6-E_4^2)$.
Recall $\eta'=\frac{1}{24}E_2 \eta$.

For $k\in\rr$, the \textit{Ramanujan-Serre differential operator} $D_k$ is defined by 
\begin{equation}
D_k f=f'-\frac{k}{12}E_2 f.
\end{equation}
This operator was originally introduced by Ramanujan \cite{ramanujan1916}.
It also appears in the theory of $p$-adic modular forms \cite{serre1973}.
We set $\srdiff{0}{k}=1$ and $\srdiff{n}{k}=D_{k+2n-2}\cdots D_{k+2}D_k$.
By Eq.\ (\ref{eq201807171108}), we have
\begin{equation}
\label{eq201712171730}
(\srdiff{n}{k}f)|_{k+2n}A=\srdiff{n}{k}(f|_k A).
\end{equation}
for $A\in\sltwoz$.
The \textit{Leibniz rule}
\begin{equation}
\label{eq201801200956}
\srdiff{n}{k+l}(fg)=\sum_{i=0}^n\binom{n}{i}(\srdiff{i}{k}f)\srdiff{n-i}{l}g
\end{equation}
will be used later.

A \textit{modular linear differential equation} (MLDE) of weight $k$ and order $n$ is an ordinary differential equation of the form
\begin{equation}
(g_0\srdiff{n}{k}+g_1\srdiff{n-1}{k}+\cdots+g_{n-1}D_k+g_n)f=0,
\end{equation}
where $g_i\in\calm(l+2i)$ for some $l\in\zz_{\ge0}$, $g_0\neq 0$ and $f$ is the unknown (holomorphic or meromorphic) function (see \cite{mason2007}).
An MLDE is called \textit{monic} if $l=0$ and $g_0=1$.
The solutions to a monic MLDE are defined on entire $\calh$ since it is simply connected.
The following result is in Theorem 4.1 in \cite{mason2007}.

\begin{prop}
\label{prop201801121402}
The solution space of a weight $k$ MLDE $(g_0\srdiff{n}{k}+g_1\srdiff{n-1}{k}+\cdots+g_{n-1}D_k+g_n)f=0$ is invariant under the modular transformations of weight $k$.
\end{prop}

\begin{proof}
Let $A\in\sltwoz$.
By the modular invariance of $g_i$ and the compatibility (\ref{eq201712171730}), we have  $((g_0\srdiff{n}{k}+g_1\srdiff{n-1}{k}+\cdots+g_{n-1}D_k+g_n)f)|_{k+l+2n}A=(g_0\srdiff{n}{k}+g_1\srdiff{n-1}{k}+\cdots+g_{n-1}D_k+g_n)(f|_kA)$.
\end{proof}

Note that $g_i$ are holomorphic functions of $q$ and $f'=q\frac{df}{dq}$.
Therefore, a monic MLDE has a singularity only at $q=0$, which turns out to be a regular singularity, hence the solutions are given by the Frobenius method.
If the indicial roots are all simple, each solution is a finite sum of $q$-series of the form $q^{\lambda}\sum_{i=0}^{\infty}a_i q^i$ (see \cite{coddington_levinson1955}).

\subsection{Generalization of Mason's theorem}

\begin{lem}
\label{masons_theorem_lemma_1}
Let $v$ be a weight $0$ multiplier on $\sltwoz$. Then the following hold.
\begin{enumerate}
\item
$\cals(0,v)=\{0\}$.
\item
$\calm(0,v)=
\begin{cases}
\cc & \text{if }v=1,\\
\{0\} & \text{if }v\neq 1.
\end{cases}$
\end{enumerate}
\end{lem}

\begin{proof}
Since the case $v=1$ is clear, assume $v\neq 1$.
Choose $i=2,4,\ldots,10$ so that $v=\chi^{2i}$.
(1) $\eta^{24-2i}\cals(0,v)\subset\cals(12-i,1)=0$.
(2) $\eta^{24-2i}\calm(0,v)\subset\cals(12-i,1)=0$.
\end{proof}

\begin{lem}
\label{masons_theorem_lemma_2}
Let $u:\sltwoz\to\cc$ be a function and $k\in\rr$ a real number.
Suppose that there exists $f\in\hol\backslash\{0\}$ such that $f|_k A=u(A)f$ for all $A\in\sltwoz$.
Then $u$ is a weight $k$ multiplier on $\sltwoz$.
\end{lem}

\begin{proof}
Since $f\neq 0$, $u(\sltwoz)\subset\cc^{\times}$.
Set $g=f/\eta^{2k}$.
Since $g|_0 A=(f|_k A) /(\eta^{2k}|_k A)\allowbreak=u(A)g/\chi^{2k}(A)$, $u/\chi^{2k}$ is a one-dimensional representation of $\sltwoz$, so $u=\chi^{2k}\chi^{2n}$ ($n=0,1,\ldots,11$) and $|u(A)|=1$ for all $A\in\sltwoz$.
The other conditions for a multiplier are clear.
\end{proof}

The following theorem generalizes Theorem 4.3 in \cite{mason2007}.

\begin{thm}
\label{thm201712301224}
Let $v$ be a weight $k\in\rr$ multiplier on $\sltwoz$, $R$ an $n$-dimensional representation and $F=(f_1,\ldots,f_n)^T\in\calm^!(k,\sltwoz,v,R)$.
Suppose that each $f_i$ has the $q$-expansion $q^{\lambda_i}\sum_{n=0}^{\infty}a_{in}q^{mn}$ with $m>0$ and $a_{i0}\neq 0$, and that $\lambda_1,\ldots,\lambda_n\in\rr$ are distinct. 
\begin{enumerate}
\item
$n(k+n-1)\ge 12(\lambda_1+\cdots+\lambda_n)$.
\item
$f_1,\ldots,f_n$ are solutions to an MLDE of order $n$ and weight $k$.
\item
The following conditions are equivalent:

\noindent
(a) $n(k+n-1)= 12(\lambda_1+\cdots+\lambda_n)$,

\noindent
(b) $f_1,\ldots,f_n$ are solutions to a monic MLDE of order $n$ and weight $k$.
\end{enumerate}
\end{thm}

\begin{proof}
(1) Set $\lambda=\lambda_1+\cdots+\lambda_n$ and $l=n(k+n-1)$.
The \textit{modular Wronskian} $W(F)=\det(F,D_k F, \srdiff{2}{k} F,\ldots,\srdiff{n-1}{k}F)$ is a nonzero $q$-series with the leading exponent $\lambda$ and the width $1/m$ since $\lambda_1,\ldots,\lambda_n$ are distinct (cf.\ Proof of Lemma 3.6 in \cite{mason2007}).
Moreover, $W(F)|_l A=v(A)^n \det(R(A))W(F)$ for all $A\in\sltwoz$ (cf.\ Lemma 3.1 and 3.4 in \cite{mason2007}).
By Lemma \ref{masons_theorem_lemma_2}, $v^n \det(R)$ is a weight $l$ multiplier on $\sltwoz$.
Therefore, $u:=\frac{v^n \det(R)}{\chi^{2l}}$ is a weight $0$ multiplier on $\sltwoz$ and so $\frac{W(F)}{\eta^{2l}}\in\calm^!(0,u)\backslash\{0\}$ has the leading exponent $\lambda-\frac{l}{12}$.
If $\lambda-\frac{l}{12}>0$, then $\frac{W(F)}{\eta^{2l}}\in\cals(0,u)\backslash\{0\}$, which contradicts Lemma \ref{masons_theorem_lemma_1}.
Therefore, $\lambda-\frac{l}{12}\le 0$ and $n(k+n-1)\ge 12(\lambda_1+\cdots+\lambda_n)$.

(3) The proof of (b)$\Rightarrow$(a) is similar to that of Theorem 4.3 in \cite{mason2007}.
We prove (a)$\Rightarrow$(b).
Since $\lambda-\frac{l}{12}=0$, we have $\frac{W(F)}{\eta^{2l}}\in\calm(0,u)\backslash\{0\}$.
By Lemma \ref{masons_theorem_lemma_1}, $u=1$ and $v^n \det(R)=\chi^{2l}$.
For each $i=1,\ldots,n$, we have
\begin{equation}
\label{masons_theorem_mlde}
\det
\begin{pmatrix}
f_i & D_k f_i & \cdots & \srdiff{n}{k} f_i\\
F & D_k F & \cdots & \srdiff{n}{k} F
\end{pmatrix}
=\sum_{j=0}^n (-1)^j W^j (F) \srdiff{j}{k} f_i=0,
\end{equation}
where $W^j (F)=\det(F,\ldots,\srdiff{j-1}{k}F,\srdiff{j+1}{k}F,\ldots,\srdiff{n}{k} F)$.
The leading exponent of $W^j (F)$ is equal to or greater than $\lambda$, and $W^j (F)|_{l+2n-2j}A=v(A)^n \det(R(A))W^j (F)$ for all $A\in\sltwoz$.
Therefore, $W^j(F)\in\calm^!(l+2n-2j,v^n \det(R))$ and $\frac{W^j(F)}{\eta^{2l}}\in\calm(2n-2j)$.
In particular, $\frac{W^n (F)}{\eta^{2l}}=\frac{W(F)}{\eta^{2l}}\in\calm(0)\backslash\{0\}=\cc^{\times}$.
By dividing the both sides of Eq.\ (\ref{masons_theorem_mlde}) by $\eta^{2l}$, we obtain
\begin{equation}
(-1)^n \frac{W(F)}{\eta^{2l}}\srdiff{n}{k} f_i +\sum_{j=0}^{n-1} (-1)^{j} \frac{W^j (F)}{\eta^{2l}} \srdiff{j}{k} f_i=0.
\end{equation}

(2) We have $\frac{W(F)}{\eta^{2l}}\in\calm^!(0,u)\backslash\{0\}$ with the leading exponent $\lambda-\frac{l}{12}$.
Set $u=\chi^{2i}$ ($i=0,2,\ldots,10$) and choose an $N\in\zz$ such that $N\ge \frac{i}{12}-\lambda+\frac{l}{12}$.
Then $\eta^{24N-2i}\frac{W(F)}{\eta^{2l}}\in\calm(12N-i)\backslash\{0\}$.
Similarly, we have $\eta^{24N-2i}\frac{W^j (F)}{\eta^{2l}}\in\calm(12N-i+2n-2j)$.
Multiplying the both sides of Eq.\ (\ref{masons_theorem_mlde}) by $\frac{\eta^{24N-2i}}{\eta^{2l}}$, we obtain
\begin{equation}
(-1)^n \eta^{24N-2i}\frac{W(F)}{\eta^{2l}}\srdiff{n}{k}f_i +\sum_{j=0}^{n-1} (-1)^{j} \eta^{24N-2i}\frac{W^j (F)}{\eta^{2l}}\srdiff{j}{k}f_i=0.
\end{equation}
\end{proof}

\section{Modular linear differential operators}
\label{section201712301238}

In this section, we define modular linear differential operators (Subsection \ref{subsection201801141056}) and show their algebraic properties (Subsections \ref{section201801110942} and \ref{subsection201801181005}).
The division properties shown in Subsection \ref{subsection201801181005} will be utilized for the proofs in the subsequent sections.

\subsection{Definition of MLDO}
\label{subsection201801141056}
In this paper, an algebra or a ring is associative and unital, but not necessarily commutative.
For any graded abelian group $A$, we denote by $A_k$ the degree $k$ homogeneous subgroup of $A$ and by $A_*$ the set of all homogeneous elements of $A$.

We set $\calm=\bigoplus_{n\in 2\zz}\calm_n$ and $\cals=\bigoplus_{n\in 2\zz}\cals_n$.
Note that $\calm=\cc [E_4,E_6]$ is a graded polynomial ring and $\cals=\Delta\calm$ is a prime ideal of $\calm$.

We set $H_n=\hol$ and $H=\bigoplus_{n\in\rr}H_n$.
For $f\in H_n$, we define the \textit{weight} of $f$ by $\wt(f)=n$.
If necessary, we denote an element of $H_n$ as $(f,n)$, where $f\in\hol$.
The space $H$ becomes a graded $\cc$-algebra with the identity $(1,0)$ by $(f,n)(g,m)=(fg,m+n)$.

We denote by $\End^i(H)$ the $\cc$-vector space $\{ \phi \in \End_{\cc}(H) \mid \phi(H_n)\subset H_{n+i}$ for all $n\in\rr\}$ and set $\End(H)=\bigoplus_{i\in\rr}\End^i(H)$, which is a graded $\cc$-algebra.
For $a\in \End^i (H)$, we define the \textit{weight} of $a$ by $\wt(a)=i$. 

Consider the following three homogeneous elements in $\End (H)$:
\begin{align}
\dd&\in\End^2 (H),\quad\dd(f,n)=(D_n f,n+2),\\
e_4&\in\End^4 (H),\quad e_4(f,n)=(E_4 f,n+4),\\
e_6&\in\End^6 (H),\quad e_6(f,n)=(E_6 f,n+6).
\end{align}
It is easy to check the commutation relations:
\begin{equation}
\label{eq201805150955}
[\dd,e_4]=-\frac{1}{3}e_6,\quad [\dd,e_6]=-\frac{1}{2}{e_4}^2,\quad [e_4,e_6]=0,
\end{equation}
where $[x,y]=xy-yx$ is the commutator.

By $\calr=\bigoplus_{i\in\zz}\calr_i$, we denote the graded $\cc$-subalgebra of $\End (H)$ generated by $\dd$, $e_4$ and $e_6$ and call it the \textit{algebra of modular linear differential operators} or the \textit{MLDO algebra}.
An element of $\calr$ is called a \textit{modular linear differential operator} (MLDO).
We remark that $H$ is a left $\calr$-module and $\calm\subset H$ is a graded $\cc$-subalgebra and a graded left $\calr$-submodule.

By $\calm'=\bigoplus_{i\in\zz}\calm'_i$, we denote the graded $\cc$-subalgebra of $\End (H)$ generated by $e_4$ and $e_6$.
We define the $\cc$-algebra homomorphism $\iota:\calm\to\calm'$ by $\iota(E_4)=e_4$ and $\iota(E_6)=e_6$.
We identify $\calm$ with $\calm'$ since $\iota$ is a grade-preserving isomorphism.
Note that an MLDE can be expressed by an MLDO as follows:
\begin{equation}
(g_n \srdiff{n}{k}+\cdots+g_1 D_k +g_0)f=0 \Leftrightarrow (g_n {\dd}^n+\cdots+g_1 \dd +g_0)(f,k)=0.
\end{equation}

For $f=\sum_{k}(f_k,k) \in H$ and $A\in\sltwoz$, we denote $\sum_k (f|_k A,k)$ by $fA$.
\begin{prop}
Let $f,g\in H$, $a\in\calr$ and $A,B\in\sltwoz$.
\begin{enumerate}
\item
\label{item201801120946}
$\dd(fg)=(\dd f)g+f(\dd g)$. (That is, $\dd$ is a derivation of $H$.)
\item
$e_4(fg)=(e_4 f)g=f(e_4 g)$ and $e_6(fg)=(e_6 f)g=f(e_6 g)$.
\item
If $f\in\bigoplus_{n\in\zz}H_n$, then $f(AB)=(fA)B$.
\item
\label{item201801120946_2}
$(fg)A=(fA)(gA)$.
\item
\label{item201801120946_3}
$(af)A=a(fA)$.
\end{enumerate}
\end{prop}

\begin{proof}
The results follow from the definitions.
\end{proof}

\subsection{Basis of the MLDO algebra}
\label{section201801110942}
In this subsection, we prove that ${e_4}^i{e_6}^j{\dd}^k$ form a $\cc$-basis of $\calr$ (Theorem \ref{thm201712301243}) and show that $\calr$ is isomorphic to a graded skew polynomial ring (Theorem \ref{thm201805311336}).

For $a\in\calr$, we set $D[a]=[\dd,a]$, $D^0[a]=a$ and $D^n[a]=D[D^{n-1}[a]]$ for $n>0$.
The Leibniz rule $D[ab]=D[a]b+aD[b]$ is clear.
By induction on $n\ge0$, 
$
D^n[ab]=\sum_{i=0}^n \binom{n}{i}D^i[a]D^{n-i}[b].
$

\begin{lem}
\label{lem201805151227}
\begin{enumerate}
\item
For $a\in\calm'$, $D[a]\in\calm'$.
\item
For $a\in\calm(k)$, $D[\iota(a)]=\iota(D_k a)$.
\item
For $a\in\calr$ and $n\in\zz_{\ge 0}$, ${\dd}^n a=\sum_{i=0}^n\binom{n}{i}D^i [a]{\dd}^{n-i}$ and $a{\dd}^n=\sum_{i=0}^n\binom{n}{i}\allowbreak (-1)^i \allowbreak {\dd}^{n-i} \allowbreak D^i[a]$.
\end{enumerate}
\end{lem}

\begin{proof}
The proof is straightforward.
\end{proof}

For $z_0\in\calh$, consider the following operator
\begin{equation}
I_{z_0}\in\End^{-2}(H),\quad(f,n)\mapsto\left(2\pi i \eta^{2n-4}(z)\int_{z_0}^z\eta^{-2n+4}(z)f(z)\,dz,n-2\right).
\end{equation}
Since $\dd I_{z_0}=1$, $\dd$ is surjective and $I_{z_0}$ is injective.
The 3-tuple $(H,\dd,I_{z_0})$ is an \textit{integro-differential algebra} 
since $R=I_{z_0}\dd$ satisfies the \textit{Rota-Baxter relation}
\begin{equation}
R(f,n)R(g,m)=(R(f,n))(g,m)+(f,n)R(g,m)-R((f,n)(g,m))
\end{equation}
(see \cite{rogensburger_rosenkranz_middeke2009}).
Note $R(f,n)=(f(z)-\eta^{2n}(z)\eta^{-2n}(z_0)f(z_0),n)$.
We also have the decomposition $H=\ker(\dd)\oplus\mathrm{im}(I_{z_0}).$

\begin{thm}
\label{thm201712301243}
The set $\{{e_4}^i {e_6}^j {\dd}^k\mid i,j,k\in\zz_{\ge 0}\}$ forms a $\cc$-basis of $\calr$.
\end{thm}

\begin{proof}
By Eq.\ (\ref{eq201805150955}), ${e_4}^i {e_6}^j {\dd}^k$ span $\calr$.
In order to show their linear independence, assume $\sum_{i,j,k\ge0}C_{ijk}{e_4}^i {e_6}^j {\dd}^k=0$.
We prove $C_{ijk}=0$ by induction on $k$.

(1) The case $k=0$.
Since $0=(\sum_{i,j,k\ge0}C_{ijk}{e_4}^i {e_6}^j {\dd}^k)(1,0)\allowbreak=\sum_{i,j\ge0}C_{ij0}{e_4}^i {e_6}^j(1,0)\allowbreak=\sum_{i,j\ge0}(C_{ij0}{E_4}^i {E_6}^j,4i+6j)$,
we have $\sum_{4i+6j=d}C_{ij0}{E_4}^i {E_6}^j \allowbreak =0$ for all $d\ge 0$.
Since $E_4$ and $E_6$ are algebraically independent, $C_{ij0}=0$ for all $i,j\ge0$.

(2) The case $k>0$.
Assume $C_{ijk'}=0$ for all $k'<k$.
Since $0=(\sum_{i,j,l\ge0}C_{ijl}{e_4}^i {e_6}^j {\dd}^l)\allowbreak(I_{z_0})^k(\eta^{4k},2k)=\sum_{i,j\ge0}C_{ijk}{e_4}^i {e_6}^j(\eta^{4k},2k)=\sum_{i,j\ge0}(C_{ijk}{E_4}^i {E_6}^j \eta^{4k},4i+6j+2k)$, we have $\sum_{4i+6j=d}C_{ijk}{E_4}^i {E_6}^j \allowbreak =0$ for all $d\ge 0$.
\end{proof}

By Theorem \ref{thm201712301243}, we have, for example, $\calr_0=\cc$, $\calr_2=\cc \dd$, $\calr_4=\cc e_4\oplus \cc {\dd}^2$, $\calr_6=\cc e_6\oplus \cc e_4 \dd\oplus \cc {\dd}^3$ and $\calr_{n}=\{0\}$ if $n<0$ or $n$ is odd.

\begin{thm}
\label{thm201712301244}
The set $\{{\dd}^k {e_4}^i {e_6}^j \mid i,j,k\in\zz_{\ge 0}\}$ forms a $\cc$-basis of $\calr$.
\end{thm}

\begin{proof}
The result follows from Lemma \ref{lem201805151227} (3) and Theorem \ref{thm201712301243}.
\end{proof}

Let $x,y,d$ be formal symbols and set $V=\cc x\oplus \cc y\oplus \cc d$.
Let $I$ be the two-sided ideal of the tensor algebra $T(V)$ generated by $[x,y]$, $[d,x]+\frac{1}{3}y$ and $[d,y]+\frac{1}{2}x\otimes x$.
Let $J$ be the left $\calr$-submodule of $\calr\otimes_{\cc}\cc$ generated by $D\otimes 1$ and $J'$ the left $\calr$-submodule of $\calr\otimes_{\cc}\cc\Delta$ generated by $D\otimes \Delta$.
Note that $\cc=\calm_0$.

Let $A_2(\cc)$ denote the Weyl algebra generated by $x$, $y$, $\partial/\partial x$ and $\partial/\partial y$.

\begin{thm}
\label{thm201801111434}
\begin{enumerate}
\item
As $\cc$-algebras, $T(V)/I$ is isomorphic to $\calr$.
\item
As left $\calr$-modules, $(\calr\otimes_{\cc}\cc)/J$ is isomorphic to $\calm$.
\item
As left $\calr$-modules, $(\calr\otimes_{\cc}\cc\Delta)/J'$ is isomorphic to $\cals$.
\item
There is a $\cc$-algebra embedding of $\calr$ in $A_2(\cc)$.
\end{enumerate}
\end{thm}

\begin{proof}
(1) The surjective $\cc$-algebra homomorphism
\begin{equation}
T(V)/I\to\calr,\quad x\mapsto e_4,\quad y\mapsto e_6,\quad d\mapsto \dd.
\end{equation}
is injective by Theorem \ref{thm201712301243}.

(2) The homomorphism
\begin{equation}
(\calr\otimes\cc)/J\to\calm,\quad a\otimes w+J\mapsto aw
\end{equation}
is clearly surjective.
To show the injectivity, let $a=\sum_{i,j,k\ge0}C_{ijk}({e_4}^i{e_6}^j {\dd}^k)\otimes 1+J\in\ker(\psi)$.
Then $\psi(a)=\sum_{i,j,k\ge0}C_{ijk}({e_4}^i{e_6}^j {\dd}^k) 1=\sum_{i,j\ge0}C_{ij0}{E_4}^i{E_6}^j=0$ and $C_{ij0}=0$ for all $i,j\ge0$, so that $a=0$.

(3) The proof is similar to that of (2).

(4) Since $\{x^iy^j(\partial/\partial x)^k(\partial/\partial y)^l \mid i,j,k,l\ge0 \}$ is a $\cc$-basis of $A_2(\cc)$, the map
\begin{equation}
\calr\to A_2(\cc),\quad e_4\mapsto x,\quad e_6\mapsto y,\quad \dd\mapsto -\frac{1}{3}y\frac{\partial}{\partial x}-\frac{1}{2}x^2\frac{\partial}{\partial y}
\end{equation}
is a $\cc$-algebra embedding.
\end{proof}

Recall the properties of \textit{skew polynomial ring} (cf.\ Chapter 1, Section 2 in \cite{mcconnell_robson2001}).
Let $R$ be a ring, $s:R\to R$ a ring homomorphism and $d:R\to R$ a $s$-derivation, that is, $d(a+b)=da+db$ and $d(ab)=(sa)db+(da)b$.
For a \textit{variable} $\xi$, the skew polynomial ring $R[\xi;s,d]$ satisfies the following properties:
(1) every element of $R[\xi;s;d]$ is uniquely expressed as a finite sum of  $r \xi^i$ with $r\in R$ and
(2) $\xi r=(sr)\xi+dr$ for $r\in R$.
When $R=\bigoplus_n R_n$ is a graded ring, $s$ preserves the grade and $d$ elevates the grade by $m$, we can equip $R[\xi;s,d]$ with a natural grading by $\wt(\xi)=m$ and $\wt(a)=n$ for $a\in R_n$ (called a \textit{graded skew polynomial ring}).

\begin{thm}
\label{thm201805311336}
As $\cc$-algebras, $\calm[\xi;1,\dd]$ is isomorphic to $\calr$.
\end{thm}

\begin{proof}
Since $\{{E_4}^i {E_6}^j \xi^k \mid i,j,k\ge0\}$ is a $\cc$-basis of $\calm[\xi;1,\dd]$, the map
\begin{equation}
\calm[\xi;1,\dd]\to\calr,\quad {E_4}^i{E_6}^j \xi^k\mapsto{e_4}^i{e_6}^j{\dd}^k
\end{equation}
is a $\cc$-linear isomorphism.
Since the commutation relation of $E_4,E_6,\xi$ is the same as that of $e_4,e_6,\dd$, we have a $\cc$-algebra isomorphism.
\end{proof}

\begin{cor}
\label{cor201804181423}
The $\cc$-algebra $\calr$ is right and left Noetherian and a right and left Ore domain.
\end{cor}

\begin{proof}
Since $\calm$ is a right and left Noetherian domain, so is $\calr=\calm[\xi;1,\dd]$ (cf.\ Theorem 1.2.9 in \cite{mcconnell_robson2001}).
A right (resp.\ left) Noetherian domain is right (resp.\ left) Ore (cf.\ Theorem 2.1.15 in \cite{mcconnell_robson2001}).
\end{proof}

\subsection{Division of MLDO}
\label{subsection201801181005}
In this subsection, we show some properties of division of two MLDOs (Theorems \ref{thm201712301245} and \ref{thm201712301246}).

By Theorem \ref{thm201712301243}, every $a\in\calr\backslash\{0\}$ can be uniquely expressed as $a_n {\dd}^n+\cdots+a_1 \dd+a_0$, where $a_i\in\calm$ and $a_n\neq 0$.
We define the \textit{top} of $a$ by $\topp(a)=a_n$ and the \textit{order} of $a$ by $\ord(a)=n$.
When $a=0$, we set $\topp(a)=0$ and $\ord(a)=-\infty$.
The conditions $a=0$, $\topp(a)=0$ and $\ord(a)=-\infty$ are equivalent to one another.
We call an MLDO $a\in\calr$ \textit{monic} when $\topp(a)\allowbreak=1$ and \textit{quasimonic} when $\topp(a)(\infty)\allowbreak=1$.
We set $\mcalr=\{a\in\calr\mid a \text{ is monic}\}$ and $\acalr=\{a\in\calr\mid a \text{ is quasimonic}\}$.
For $i\in[-\infty,\infty]$, we set
\begin{align}
\calr^i&=\{a\in\calr\mid \ord(a)=i\},&\calr_n^i&=\calr^i\cap\calr_n, \label{eq201806240958} \\
\calr^{\le i}&=\{a\in\calr\mid \ord(a)\le i\}, &\calr_n^{\le i}&=\calr^{\le i}\cap\calr_n,\\
\calr^{<i}&=\{a\in\calr\mid \ord(a)<i\},& \calr_n^{<i}&=\calr^{<i}\cap\calr_n, \label{eq201806240958_3}
\end{align}
\begin{align}
\mcalr^i&=\mcalr\cap\calr^i, & \mcalr_n&=\mcalr\cap\calr_n, & \mcalr_n^i&=\mcalr\cap\calr_n^i,\\
\acalr^i&=\acalr\cap\calr^i, & \acalr_n&=\acalr\cap\calr_n, & \acalr_n^i&=\acalr\cap\calr_n^i.
\end{align}

\begin{prop}
\label{prop201801051831}
For $a,b\in\calr$, $\topp(ab)=\topp(a)\topp(b)$ and $\ord(ab)=\ord(a)+\ord(b)$.
\end{prop}

\begin{proof}
The result follows from Lemma \ref{lem201805151227} (3).
\end{proof}

\begin{thm}
\label{thm201712301245}
(Division of MLDO, inhomogeneous version)

Let $a\in\calr^i$ and $b\in\calr^j$ with $i\ge j\ge 0$, and let $d\in\calm$ be a common divisor of $\topp(a)$ and $\topp(b)$.
Let $a',b'\in\calm$ satisfy $\topp(a)=a'd$ and $\topp(b)=b'd$.
\begin{enumerate}
\item
$\topp(b)^{i-j}b'a=cb+c'$ for some $c\in\calr^{i-j}$ and $c'\in\calr^{<j}$.
\item
$a\topp(b)^{i-j}b'=bc+c'$ for some $c\in\calr^{i-j}$ and $c'\in\calr^{<j}$.
\end{enumerate}
\end{thm}

\begin{proof}
(1) Note that $d$, $a'$ and $b'$ cannot be $0$.
We prove (1) by induction on $i$.

(I) The case $i=j$.
Set $c=a'\in\calr^0$ and $c'=b'a-cb\in\calr^{<j}$.

(II) The case $i>j$.
Set $e=b'a-a'{\dd}^{i-j}b\in\calr^{<i}$.

(II-i) The case $e\in\calr^{<j}$.
Set $c=\topp(b)^{i-j}a'{\dd}^{i-j}\in\calr^{i-j}$ and $c'=\topp(b)^{i-j}e\in\calr^{<j}$.

(II-ii) The case $e\in\calr^{j'}$ with $i> j' \ge j$.
Set $k=i-j'-1\ge 0$.
By the induction hypothesis, $\topp(b)^{j'-j+1}e=fb+g$ for some $g\in\calr^{<j}$ and $f\in\calr^{j'-j}$.
We have $\topp(b)^{i-j}b'a
=(\topp(b)^{i-j}a'{\dd}^{i-j}+\topp(b)^k f)b+\topp(b)^k g$.
Set $c=\topp(b)^{i-j}a'{\dd}^{i-j}+\topp(b)^k f\in\calr^{i-j}$ and $c'=\topp(b)^k g\in\calr^{<j}$.

(2) We only give a sketch.
(I) Set $c=a'$ and $c'=ab'-ba'$.
(II) Set $e=ab'-ba'{\dd}^{i-j}$.
(II-i) Set $c=a'{\dd}^{i-j}\topp(b)^{i-j}$ and $c'=e\topp(b)^{i-j}$.
(II-ii) Set $c=a'{\dd}^{i-j}\topp(b)^{i-j}+f\topp(b)^k$ and $c'=g\topp(b)^k$.
\end{proof}

\begin{thm}
\label{thm201712301246}
(Division of MLDO, homogeneous version)

Let $a\in\calr_n^i$ and $b\in\calr_m^j$ with $i\ge j\ge 0$ and $m,n\ge 0$, and let $d\in\calm_l$ be a common divisor of $\topp(a)\in\calm_{n-2i}$ and $\topp(b)\in\calm_{m-2j}$.
Let $a'\in\calm_{n-2i-l}$ and $b'\in\calm_{m-2j-l}$ satisfy $\topp(a)=a'd$ and $\topp(b)=b'd$.
Then, for $p=(m-2j)(i-j+1)-l-m+n$ and $q=(m-2j)(i-j+1)-l+n$, the following hold:
\begin{enumerate}
\item
$\topp(b)^{i-j}b'a=cb+c'$ for some $c\in\calr_{p}^{i-j}$ and $c'\in\calr_{q}^{<j}$,
\item
$a\topp(b)^{i-j}b'=cb+c'$ for some $c\in\calr_{p}^{i-j}$ and $c'\in\calr_{q}^{<j}$.
\end{enumerate}
\end{thm}

\begin{proof}
Trace the proof of the inhomogeneous version.
Note that in (II-ii), $\wt(e)=(m-2j)-l+n$, $\wt(f)=(m-2j)(j'-j+2)-l-m+n$ and $\wt(g)=(m-2j)(j'-j+2)-l+n$.
\end{proof}

\begin{cor}
\label{cor201802231435}
Let $a\in\calr$ and $b\in\calr^j$ with $\topp(b)=1$.
\begin{enumerate}
\item
$a=cb+c'$ for some $c\in\calr$ and $c'\in\calr^{<j}$.
\item
$a=bc+c'$ for some $c\in\calr$ and $c'\in\calr^{<j}$.
\end{enumerate}
\end{cor}

\begin{proof}
We only prove (1).
Set $a=\sum_{i} a_i+\sum_{i} \tilde{a}_i$, where $a_i\in\calr^{n_i}$, $\tilde{a}_i\in\calr^{m_i}$ and $n_i<j\le m_i$.
Then $\tilde{a}_i=c_i b+c'_i$ for some $c_i\in\calr^{m_i-j}$ and $c'_j\in\calr^{<j}$.
Set $c=\sum_{i} c_i\in\calr$ and $c'=\sum_{i} a_i +\sum_{i} c'_i\calr^{<j}$.
\end{proof}

\begin{cor}
Let $a\in\calr_k$ and $b\in\calr^j_{2j}$ with $\topp(b)=1$.
\begin{enumerate}
\item
$a=cb+c'$ for some $c\in\calr_{k-2j}$ and $c'\in\calr^{<j}_{k}$.
\item
$a=bc+c'$ for some $c\in\calr_{k-2j}$ and $c'\in\calr^{<j}_{k}$.
\end{enumerate}
\end{cor}

\begin{proof}
The proof is similar to that of Corollary \ref{cor201802231435}.
\end{proof}

Let $a\in\calr$ and $b\in\calr\backslash\{0\}$.
If $a=cb$ (resp.\ $a=bc$) for some $c\in\calr$, we say that $a$ is \textit{divisible} by $b$ from the right (resp.\ left).  
Such $c$ is unique since $\calr$ is an integral domain.
We denote it by $a/b$ (resp.\ $b\backslash a$).

\begin{rem}
All the results in this subsection hold for a graded skew polynomial ring $S=R[x;1,d]$, where $R$ is a graded commutative integral domain and $d$ is a graded derivation of $R$.
\end{rem}

\section{Algebraic structure of the solutions to MLDEs}
\label{section20180614}

In this section, we study an algebraic structure of the solutions to MLDEs.
We also show that every quasimodular form satisfies a monic MLDE of some weight.
Subsection \ref{subsct201806231428} is devoted to the monic case and \ref{subsct201806231430} to the non-monic case.

\subsection{Monic case}
\label{subsct201806231428}
We set:
\begin{align}
S_k&=\{(f,k)\in H_k\mid a(f,k)=0\text{ for some }a\in\mcalr\},\\
\Sigma_k&=\{f \in\hol\mid f\text{ satisfies some monic MLDE of weight }k\}.
\end{align}
Note that $f \in \Sigma_k$ if and only if $(f,k)\in S_k$.

\begin{lem}
\label{lem201804181457}
For $a\in\calr$ and $b\in\mcalr$, there exist $a'\in\mcalr$ and $b'\in\calr$ such that $a'a=b'b$.
\end{lem}

\begin{proof}
By Lemma 1.2 in \cite{jategaonkar72} or Proposition 2.2 in \cite{resco_small_stafford82}, $\mcalr$ is a left Ore set in $\calr$.
\end{proof}

\begin{lem}
\label{lem201804240917}
For $\phi\in H_k$, the following conditions are equivalent.
\begin{enumerate}
\item
$\phi\in S_k$.
\item
The $\calm$-module $\calr \phi=\sum_{i=0}^{\infty}\calm {\dd}^i \phi$ is finitely generated.
\item
There exists $n\in\zz_{\ge 0}$ such that $\phi,\dd \phi,\ldots,{\dd}^n\phi$ generate $\calr \phi$ over $\calm$.
\end{enumerate}
\end{lem}

\begin{proof}
The equivalence of (1) and (2) is easy (see Lemma 2.1 in \cite{resco_small_stafford82}) and it is clear that (3) implies (2).

Assume (1).
Then $a \phi=0$ for some $a\in\mcalr^n$.
Therefore, $\phi,\dd \phi,\ldots,{\dd}^{n-1}\phi$ generate $\calr \phi$ over $\calm$.
\end{proof}

\begin{thm}
\label{thm201804191022}
\begin{enumerate}
\item
$S_k+S_k\subset S_k$.
\item
$S_k S_l\subset S_{k+l}$.
\item
$R_k S_l\subset S_{k+l}$.
\item
$\calm_k\subset S_k$.
\end{enumerate}
Therefore, $\bigoplus_{k\in\rr}S_k\subset H$ is a $\cc$-subalgebra and a left $\calr$-submodule.
\end{thm}

\begin{proof}
Let $\phi\in S_k$ and $\psi\in S_l$.
Choose $m,n\ge 0$ such that $\phi,\dd \phi,\ldots, {\dd}^m \phi$ (resp.\ $\psi,\dd \psi,\ldots, {\dd}^n \psi$) generate $\calr \phi$ (resp.\ $\calr\psi$) over $\calm$.

(1) Assume $l=k$.
We have $\calr(\phi+\psi)\subset\calr\phi+\calr\psi$, where the latter is generated by $\phi,\dd \phi,\ldots, {\dd}^m \phi$ and $\psi,\dd \psi,\ldots, {\dd}^n \psi$.
Since $\calm$ is Noetherian, $\calr(\phi+\psi)$ is finitely generated.

(2) By the Leibniz rule, ${\dd}^p(\phi\psi)=\sum_{q=0}^p\binom{p}{q}({\dd}^q\phi){\dd}^{p-q}\psi$ is a finite sum of $({\dd}^i\phi){\dd}^j\psi$ over $\calm$, where $0\le i\le m$ and $0\le j\le n$.
Thus $\calr(\phi\psi)\subset \sum_{i=0}^{m-1}\sum_{j=0}^{n-1}\calm ({\dd}^i\phi){\dd}^j\psi$.

(3) $\calr(a \psi)\subset \calr\psi$ for $a\in R_k$.

(4) Since $(1,0)\in S_0$ and $\calm_k\subset R_k$, the result follows from (3).

We can also prove (1) and (3) by utilizing Lemma \ref{lem201804181457}.
\end{proof}

We apply Theorem \ref{thm201804191022} to the theory of monic MLDEs.
We set $Y^i=(1,-i)\in H_{-i}$.
Then $\dd Y^i=\frac{i}{12}E_2 Y^i$.
Note $E_2\in H_2$, $E_4\in H_4$ and $E_6\in H_6$.

\begin{thm}
\label{thm201804240953}
\begin{enumerate}
\item
$\Sigma_k\subset \Sigma_{k-1}$.
\item
$(\log q)\Sigma_k\subset \Sigma_{k-1}$.
\item
$E_2 \Sigma_k\subset \Sigma_{k+1}$.
\item
If $k-l\in\zz$, then $\Sigma_k+\Sigma_l\subset \Sigma_{\min\{k,l\}}$.
\end{enumerate}
\end{thm}

\begin{proof}
(1) (2) $Y$ and $(\log q)Y$ are annihilated by ${\dd}^2+\frac{1}{144}e_4$.
(3) $E_2 Y$ is annihilated by ${\dd}^3-\frac{23}{144}e_4 \dd-\frac{1}{216}e_6$.
(4) The result follows from (1).
\end{proof}

Set $\calqm=\cc[E_2,E_4,E_6]\subset H$.
Recall that a \textit{quasimodular form} is a homogeneous element of $\calqm$ and its \textit{depth} is the maximum degree in $E_2$.
For example, ${E_2}^2 E_4+3 E_2 E_6-2 {E_4}^2$ is a quasimodular form of weight $8$ and depth $2$.
We denote by $\calqm_k^{\le s}$ the set of quasimodular forms of weight $k$ and depth at most $s$.
We formally set $\calqm_k^{\le s}=0$ for $s\in\zz_{<0}$.
Note that $\calm_k=\calqm_k^{\le 0}$, $\calqm_k^{\le s}=\bigoplus_{i=0}^s\calm_{k-2i}{E_2}^i$ and $\dd(\calqm_k^{\le s})\subset \calqm_{k+2}^{\le s+1}$ for $s\in\zz$.

For $f\in\hol$, we denote by $\dwt(f)\in\zz\cup\{\infty\}$ the largest $k\in\zz$ such that $f\in \Sigma_k$.
When $f$ does not satisfy a monic MLDE of any integral weight, we set $\dwt(f)=-\infty$.

\begin{thm}
\label{thm201804241015}
\begin{enumerate}
\item
$\dwt(fg)\ge\dwt(g)+\dwt(g)$, where we formally set $\infty+(-\infty):=-\infty$.
\item
If $\dwt(f)=\dwt(g)$, then $\dwt(f+g)\ge\dwt(g)$.
\item
If $\dwt(f)>\dwt(g)$, then $\dwt(f+g)=\dwt(g)$.
\item
If $f$ is a nonzero quasimodular form of weight $k$ and depth $s$, then $\dwt(f)=k-s$.
\end{enumerate}
\end{thm}

\begin{proof}
(1) (2) Straightforward.

(3) If $-\infty<\dwt(g)<\infty$, then $\dwt(f+g)\ge \dwt(g)$.
Since $f\in \Sigma_{\dwt(g)+1}$, $\dwt(f+g)\ge\dwt(g)+1$ implies $\dwt(g)\ge\dwt(g)+1$, which is a contradiction.
The case $\dwt(g)=-\infty$ is clear.
 
(4) Assume $f={E_2}^s f'$, where $f'\in\calm_{k'}\backslash\{0\}$ and $k'=k-2s$.
By Theorem \ref{thm201804240953} (3), $f\in \Sigma_{k-s}$.
By induction on $n\ge 0$, ${\dd}^n(Y^{s-1}{E_2}^{s}f')=Y^{s-1}(\frac{n!}{(-12)^n}{E_2}^{n+s}f'+g)$ for some $g\in\calqm_{k'+2n+2s}^{\le n+s-1}$.
Since $E_2$, $E_4$ and $E_6$ are algebraically independent and $f'\in\cc[E_4,E_6]\backslash\{0\}$, a nonzero MLDO never annihilates $Y^{s-1}{E_2}^{s}f'$.
Therefore, $\dwt(f)=k-s$.
For general $f$, express it as $f=\sum_{i=0}^s{E_2}^i f_i$ with $f_i\in\calm_{k-2i}$ and $f_s\neq 0$.
The result follows from (3).
\end{proof}

\begin{rem}
By Theorem \ref{thm201804241015}, every quasimodular form of weight $k$ and depth $s>0$ satisfies a monic MLDE of weight $k-s$, which is lower than the \textit{original} weight $k$.
This phenomenon has already been found in the study of Kaneko-Zagier equation.
See the paragraph just after the proof of Theorem 1 in \cite{kaneko_koike2003}.
See also Subsection 3.3 in \cite{kaneko_nagatomo_sakai2017}.
\end{rem}

\begin{rem}
Let us generalize Theorem \ref{thm201804191022}.
Let $G$ be an abelian group, $(R'=\bigoplus_{g\in G}R'_g,D)$ be a graded differential commutative ring and ($R=\bigoplus_{g\in G}R_g,D)\subset R'$ a graded differential subring.
Denote by $T_g$ the set of all $s\in R'_g$ satisfying $(D^n+\cdots+r_1 D+r_0)s=0$ for some $r_i\in R$.
Assume that $R$ is graded-Noetherian, that is, every graded ideal of $R$ is finitely generated.
(For graded rings and modules, see, for example, \cite{nastasescu_oystaeyen2004}.)
Then $\bigoplus_{g\in G}T_g$ is a graded differential ring containing $R$.

By Theorem \ref{thm201804241015}, $E_2$ does not satisfy a monic MLDE of weight $2$.
Consider the setting $G=\rr$, $R'=H$, $R=\calqm$ and $D=\dd$.
Since $R$ is Noetherian, $E_2$ satisfies a monic linear differential equation of weight $2$ with coefficients in $\calqm$.
Indeed, $(\srdiff{3}{2}+x E_2 \srdiff{2}{2}+(y {E_2}^2-\frac{13}{72} E_4)D_2+\frac{1}{288}(1-4x+24y) {E_2}^3+\frac{1}{288}(-3-4x+24y)E_2E_4+\frac{1}{216}(1-6x)E_6)E_2=0$ for every $x,y\in\cc$.
\end{rem}

\subsection{Non-monic case}
\label{subsct201806231430}

We set $M=\bigoplus_{k\in\rr}M_k=\bigoplus_{k\in\rr}\mer$ and $\ocalm=\{f/g\in\mer\mid f\in\calm,g\in\calm_*\backslash\{0\}\}$, where the latter is the set of \textit{meromorphic modular forms} and has a natural grading by $\wt(f/g)=\wt(f)-\wt(g)$ for $f,g\in\calm_*$ and $g\neq 0$.
It is obvious that $\ocalm\subset M$ are \textit{graded fields}, that is, they are nonzero graded commutative rings and every nonzero homogeneous element is invertible (cf.\ \cite{geel_oystaeyen1981} and \cite{nastasescu_oystaeyen2004}, for example).
Note that a graded field is not necessarily a field although a graded ring is always a ring. 

We denote by $\ocalr=\bigoplus_{k\in\zz}\ocalr_k$ the graded subring of $\End(M)$ generated by $\dd$ and $\ocalm$, and call an element of $\ocalr$ a \textit{meromorphic MLDO} or a \textit{merMLDO} for short.
Every merMLDO can be uniquely expressed as a finite sum of $f {\dd}^i$ with $f\in\ocalm$.
The order, top and monicness of a merMLDO are defined as with those of an MLDO.
We also define $\ocalr^n$, $\ocalr^n_k$ etc.\ in the same way as Eqs.\ \ref{eq201806240958} through \ref{eq201806240958_3}. 
We have $\ocalr\simeq\ocalm[\xi;1,\dd]$.

Note that all the results in \cite{ore1933} are applicable to $\ocalr$.
For example, $\ocalr$ is a left (resp.\ right) graded-PID, that is, $\ocalr$ is an integral domain and every left (resp.\ right) graded ideal of $\ocalr$ is generated by one homogeneous element.

For $a,b\in\ocalr_*\backslash\{0\}$, we denote by $\gcrd(a,b)$ (called the \textit{greatest common right divisor}) the monic $c\in\ocalr_*$ such that $c$ divides $a,b$ from the right and if $c'\in\ocalr_*$ divides $a,b$ from the right, then $c'$ divides $c$ from the right.
We also denote by $\lclm(a,b)$ (called the \textit{least common left multiple}) the monic $c\in\ocalr_*$ such that $c$ is divided by $a,b$ from the right and if $c'\in\ocalr_*$ is divided by $a,b$ from the right, then $c'$ is divided by $c$ from the right.
By Chapter I, Sections 2 and 3 in \cite{ore1933}, $\gcrd(a,b)$ and $\lclm(a,b)$ always exist uniquely, $\ord(\gcrd(a,b))+\ord(\lclm(a,b))=\ord(a)+\ord(b)$ and $a'a+b'b=\gcrd(a,b)$ for some $a',b'\in\ocalr_*$.
The proof of the following proposition is straightforward.

\begin{prop}
Let $\phi\in M_*$ and $a,b\in\ocalr_*\backslash\{0\}$.
\begin{enumerate}
\item
$\gcrd(a,b)\phi=0$ if and only if $a\phi=b\phi=0$.
\item
$\lclm(a,b)\phi=0$ if $a\phi=0$ or $b\phi=0$.
\end{enumerate}
\end{prop}

For $n\in\zz_{\ge 0}$ we set:
\begin{align}
\overs^n_k&=\{(f,k)\in M_k\mid a(f,k)=0\text{ for some }a\in\ocalr^n\},\\
\osigma^n_k&=\{f \in\mer\mid f\text{ satisfies some MLDE of weight }k\text{ and order }n\}.
\end{align}
It is clear that $f \in \osigma^n_k$ if and only if $(f,k)\in \overs^n_k$.
We have $\overs^n_k\subset \overs^{n+1}_k$ and $\osigma^n_k\subset \osigma^{n+1}_k$.

Note the following facts.
Let $A$ be a graded field and $N\neq 0$ a graded $A$-module.
If $\emptyset \neq N_1\subset N_2\subset N_*$ such that $N_1$ is linearly independent and $N_2$ generates $N$ over $A$, then $N$ has an $A$-basis $N'$ such that $N_1\subset N'\subset N_2$.
The cardinality of a basis is independent of the choice of bases since $A$ is a commutative ring.

\begin{lem}
For $\phi\in M_k$, the following conditions are equivalent.
\begin{enumerate}
\item
$\phi\in\overs^n_k$.
\item
$\phi,\dd\phi,\ldots,{\dd}^{n-1}\phi$ generate $\ocalr\phi$ over $\ocalm$.
\item
$\dim_{\ocalm}\ocalr \phi\le n$.
\end{enumerate}
\end{lem}

\begin{proof}
It is straightforward to prove $(1)\Leftrightarrow (2)\Rightarrow (3)$.
Assume (3).
Then $\phi,\ldots,\allowbreak{\dd}^{n}\phi$ are linearly dependent over $\ocalm$.
Therefore, $\phi\in\overs^m_k$ for some $m\le n$, so $\phi\in\overs^n_k$ and (1) holds.
\end{proof}

\begin{thm}
\label{thm20180624}
\begin{enumerate}
\item
$\overs^n_k+\overs^m_k\subset \overs^{n+m}_k$.
\item
$\overs^n_k\overs^m_l\subset \overs^{nm}_{k+l}$.
\item
$\ocalr_k \overs_l^m\subset \overs_{k+l}^m$.
\item
$\ocalm_k\subset \overs_k^1$.
\end{enumerate}
\end{thm}

\begin{proof}
Let $\phi\in\overs^n_k$ and $\psi\in\overs^m_l$.

(1) Assume $k=l$.
We have $a\phi=b\psi=0$ for some $a\in\ocalr_*^n$ and $b\in\ocalr_*^m$, so $\lclm(a,b)(\phi+\psi)=0$.
Note $\ord(\lclm(a,b))\le \ord(a)+\ord(b)=n+m$.

(2) Since $\ocalr (\phi\psi)$ is generated by $({\dd}^i\phi){\dd}^j\psi$ over $\ocalm$ with $0\le i\le n-1$ and $0\le j \le m-1$, $\dim_{\ocalm}\ocalr (\phi\psi)\le nm$.

(3) $\ocalr(a\psi)\subset \ocalr\psi$ for $a\in\ocalr_k$.

(4) $(f D_{k}-(D_k f))f=0$ for $f\in\ocalm_k$.
\end{proof}

\begin{rem}
We have $\Sigma_k\subset \hol\cap (\bigcup_{n=0}^{\infty} \osigma^n_k)$, but the equality does not hold.
Let $g=q^{l}+\cdots\in\calm_{12m}$ with $m\in\zz_{>0}$ and $l<m$ (for example, $g={E_4}^3=1+\cdots\in\calm_{12}$) and $f=\exp(g/\Delta^m)\in\hol$.
Since $(\Delta^m D_0-D_{12m}g)f=0$, $f\in \hol\cap\osigma_0^1$.
Since $f$ has an essential singularity at $q=0$, it cannot be constructed by the Frobenius method and $f\notin \Sigma_k$ for any $k\in\rr$.
(We can also show $f\notin \Sigma_k$ by $\srdiff{n}{k}f=((-m+l)^n q^{-n(m-l)}+\cdots)f$.)
Therefore, $f\in \hol\cap (\bigcup_{n=0}^{\infty} \osigma^n_0)\backslash \Sigma_0$.
It follows that $\eta^{2k} f\in\hol\cap (\bigcup_{n=0}^{\infty} \osigma^n_k)\backslash \Sigma_k$ for $k\in\rr$.
\end{rem}

\section{Operators on solution spaces of monic MLDEs}
\label{section201712301252}
We define the monic and quasimonic MLDOs (Subsection \ref{subsection201801081050}) and give a condition under which an MLDO maps the solution space of a monic MLDE to that of another (Subsection \ref{subsection201804181356}, Corollary \ref{cor201801191507}).
We then introduce a family of third order monic MLDEs and apply the results to the solution spaces (Subsections \ref{sct201801121359} and \ref{subsection201804181403}, Example \ref{ex201807171044}).

\subsection{Monic and quasimonic MLDOs}
\label{subsection201801081050}
Recall that an MLDO $a\in\calr$ is called monic when $\topp(a)\allowbreak=1$ and quasimonic when $\topp(a)(\infty)\allowbreak=1$.

Let $a\in\calr$ and $k\in\zz$.
We define the $\cc$-linear map
\begin{equation}
a[k]:\hol\to\hol,\quad a[k]f=\pi(a(f,k)),
\end{equation}
where $\pi(\sum_{n} (g_n,n))= \sum_{n}g_n$.
The map $a\mapsto a[k]$ is $\cc$-linear.

For each $a\in\calr$, there exists a unique $c$ such that $a[k]q^{\lambda}=cq^{\lambda}+O(q^{\lambda+1})$.
We denote such $c$ by $F(k,a,\lambda)$.
It is a polynomial in $\lambda$ and linear in $a$.

\begin{lem}
If $a=\sum_{i=0}^n a_i {\dd}^i$ with $a_i\in\calm$, then $a[k]=\sum_{i=0}^n a_i \srdiff{i}{k}$ and $F(k,a,\lambda)=\sum_{i=0}^n a_i(\infty)P_i(\lambda),$ where $P_0(\lambda)=1$ and $P_i(\lambda)=(\lambda-\frac{k+2i-2}{12})\cdots(\lambda-\frac{k}{12})$ for $i>0$.
\end{lem}

\begin{proof}
The proof is easy.
\end{proof}

The equality $F(k,a,\lambda)=F(0,a,\lambda-\frac{k}{12})$ holds since $D_k q^{\lambda}=(q\frac{d}{dq}-\frac{k}{12}E_2)q^{\lambda}=(\lambda-\frac{k}{12})q^{\lambda}+O(q^{\lambda+1})$.

We set
\begin{equation}
\label{eq20180707}
Z=\{\sum_{i=0}^n a_i {\dd}^i \in\calr \mid n\ge 0,\ a_i\in\calm,\ a_i(\infty)=0\}.
\end{equation}
Note that $Z=\{a\in\calr\mid F(k,a,\lambda)\equiv0\text{ for all }k\in\zz_{\ge 0}\}=\{a\in\calr\mid F(k,a,\lambda)\equiv0\text{ for some }k\in\zz_{\ge 0}\}$.
It is straightforward to show $Z\cap\calr_l=\Delta\calr_{l-12}$ and $\Delta\calr\subset Z$.

\begin{prop}
\label{prop201801091422}
For $k,n\in\zz_{\ge0}$, the following hold.
\begin{enumerate}
\item
$\mcalr_{2n}^n=\acalr_{2n}^n$.
\item
$\mcalr\cap Z=\emptyset$ and $\acalr\cap Z=\emptyset$.
\item
$\mcalr_k^n\neq\emptyset$ if and only if $k=2n$.
\item
\label{item201801191423}
$\acalr_k^n\neq\emptyset$ if and only if $k=2n$ or $k=2n+4,2n+6,\ldots$.
\end{enumerate}
\end{prop}

\begin{proof}
The proof is straightforward.
\end{proof}

\begin{lem}
For $a\in\calr_l$ and $b\in\calr$, $(ba)[k]=b[k+l]a[k]$.
\end{lem}

\begin{proof}
For $f\in\hol$, $b[k+l]a[k]f=b[k+l]\pi(a(f,k))=\pi(ba(f,k))=(ba)[k]f$.
\end{proof}

\begin{lem}
For $a\in\calr_k$ and $b\in\calr$, $F(l,ba,\lambda)=F(l+k,b,\lambda)F(l,a,\lambda)$.
\end{lem}

\begin{proof}
$(ba)[l]q^{\lambda}=b[l+k]a[k]q^{\lambda}=b[l+k](F(k,a)q^{\lambda}+O(q^{\lambda+1}))=F(l+k,b)F(l,a)q^{\lambda}+O(q^{\lambda+1})$.
\end{proof}

\begin{thm}
\label{proposition_mldo_divisible_from_the_right}
Let $a\in\mcalr_{2n}^n$, $b\in\calr$ and $k\in\zz_{\ge 0}$.
Then the following conditions are equivalent:
\begin{enumerate}
\item
$\ker(a[k])\subset \ker(b[k])$,
\item
$b$ is divisible by $a$ from the right.
\end{enumerate}
\end{thm}

\begin{proof}
If $b=ca$ with some $c\in\calr$, then $b[k]=c[k+2n]a[k]$ and $\ker(a[k])\subset \ker(b[k])$.

Assume (1).
By Corollary \ref{cor201802231435}, $b=ca+c'$ for some $c\in\calr$ and $c'\in\calr^{<n}$.
Let $f\in\ker(a[k])$.
We have $a[k]f=0$ and $b[k]f=0$.
Since $b[k]=c[k+2n]a[k]+c'[k]$, we have $c'[k]f=0$ and $\ker(a[k])\subset\ker(c'[k])$.
If $c'\neq 0$, then $\dim\ker(c'[k])<n=\dim\ker(a[k])$, which is a contradiction.
Therefore, $c'=0$ and $b=ca$.
\end{proof}

\subsection{Characteristic roots of MLDOs}
\label{subsection201804181356}

For $a \in\calr$, $F(k,a,\lambda)=0$ is an algebraic equation with the variable $\lambda$.
We denote the multiset (unordered tuple) of the roots (\textit{characteristic roots}) by $\ch(k,a)$.
When $F(k,a,\lambda)\equiv 0$, we set $\ch(k,a)=\cc$.
If $\ch(0,a)=\{\lambda_1,\ldots,\lambda_n\}$, then $\ch(k,a)=\{\lambda_1+\frac{k}{12},\ldots,\lambda_n+\frac{k}{12}\}$.
Recall the set $Z$ defined in Eq.\ (\ref{eq20180707}).

\begin{lem}
Let $a,b\in\acalr$.
Then $\ch(k,a)=\ch(k,b)$ if and only if $a-b\in Z$.
\end{lem}

\begin{proof}
Assume $a-b\in Z$.
Then $F(k,a,\lambda)-F(k,b,\lambda)=F(k,a-b,\lambda)=0$ and $\ch(k,a)=\ch(k,b)$.

Assume $\ch(k,a)=\ch(k,b)$. Since $a,b\notin Z$, $F(k,a,\lambda)$ and $F(k,b,\lambda)$ are nonzero monic polynomials in $\lambda$.
Thus, $F(k,a,\lambda)=F(k,b,\lambda)$, $0=F(k,a,\lambda)-F(k,b,\lambda)=F(k,a-b,\lambda)$ and so $a-b\in Z$.
\end{proof}

\begin{lem}
\label{lem201712130947}
If $a\in\mcalr^n_{2n}$, then the sum of $\lambda\in\ch(k,a)$ is $\frac{n(k+n-1)}{12}$.
\end{lem}

\begin{proof}
Set $a={\dd}^n+\sum_{i=0}^{n-1} a_i {\dd}^i$ with $a_i\in\calm_{2n-2i}$.
We have $a_{n-1}=0$ since $\calm_2=\{0\}$.
Therefore, $a[k]=\srdiff{n}{k}+a_{n-2}\srdiff{n-2}{k}+\cdots$, so that $F(k,a,\lambda)=(\lambda-\frac{k}{12})\cdots(\lambda-\frac{k+2n-2}{12})+O(\lambda^{n-2})=\lambda^n-\frac{n(k+n-1)}{12}\lambda^{n-1}+O(\lambda^{n-2})$.
\end{proof}

\begin{lem}
\label{lem201712130954}
Let $k\in\zz_{\ge0}$ and $n\in\zz_{>0}$.
Then the following hold.
\begin{enumerate}
\item
For $\lambda_1,\ldots,\lambda_n\in\cc$ and $l\in\zz$ such that $l\ge 2n+4$, there exists $a\in\acalr_{l}^n$ such that $\ch(k,a)=\{\lambda_1,\ldots,\lambda_n\}$.
\item
For $\lambda_1,\ldots,\lambda_n\in\cc$ such that $\lambda_1+\cdots+\lambda_n=\frac{n}{12}(k+n-1)$, there exists $a\in\mcalr_{2n}^n$ such that $\ch(k,a)=\{\lambda_1,\ldots,\lambda_n\}$.
\end{enumerate}
\end{lem}

\begin{proof}
Set $P_0(\lambda)=1$ and $P_i(\lambda)=(\lambda-\frac{k}{12})\cdots(\lambda-\frac{k+2i-2}{12})$ for $i>0$.

(1) Choose $x_0,\ldots,x_{n-1}\in\cc$ so that the roots of $P_n(\lambda)+\sum_{i=0}^{n-1}x_i P_i(\lambda)=0$ are $\lambda=\lambda_1,\ldots,\lambda_n$.
If $N\in\zz_{\ge4}$, then $f(\infty)=1$ for some $f\in\calm_N$.
Therefore, for $0\le i\le n$, we can choose $a_i\in\calm_{l-2i}$ such that $a_i(\infty)=1$.
Set $a=a_n {\dd}^n+\sum_{i=0}^{n-1}x_i a_i {\dd}^i\in\acalr_l^n$.

(2) Since $\lambda_1+\cdots+\lambda_n=\frac{n}{12}(k+n-1)$, there exist $x_0,\ldots,x_{n-2}\in\cc$ so that the roots of $P_n(\lambda)+\sum_{i=0}^{n-2}x_i P_i(\lambda)=0$ are $\lambda=\lambda_1,\ldots,\lambda_n$.
For $0\le i\le n-2$, choose $a_i\in\calm_{2n-2i}$ such that $a_i(\infty)=1$ and set $a={\dd}^n+\sum_{i=0}^{n-2}x_i a_i {\dd}^i\in\mcalr_{2n}^n$.
\end{proof}

\begin{thm}
\label{thm201801191511}
Let $l,n,N\in\zz_{>0}$, $a\in\mcalr^n_{2n}$ and $b\in\acalr_l$.
Then the following hold.
\begin{enumerate}
\item
If $N\le n$ and there exists $c\in\mcalr_{2N}^N$ such that $cb$ is divisible by $a$ from the right, then $\ch(0,a)\cap\ch(0,b)\neq\emptyset$.
\item
If $\ch(0,a)\cap\ch(0,b)\neq\emptyset$, $N\le|\ch(0,a)\cap\ch(0,b)|$ and $l+2n-2N\le8$, then there exists $c\in\mcalr_{2n-2N+2}^{n-N+1}$ such that $cb$ is divisible by $a$ from the right.
\end{enumerate}
\end{thm}

\begin{proof}
Note that $\ch(\bullet,\bullet)$ is a multiset, not a set.
We assume $b\in\acalr^m_l$ for some $m\in\zz$.
By Proposition \ref{prop201801091422} (\ref{item201801191423}), $l=2m,2m+4,2m+6,\ldots$.

(1) $cb=da$ for some $d\in\calr_{l-2n+2N}^{m-n+N}$.
By comparing the tops of the both sides, $d\in\acalr_{l-2n+2N}^{m-n+N}$.
Thus, $\ch(l,c)\cup\ch(0,b)=\ch(2n,d)\cup\ch(0,a)$.
If $\ch(0,a)\cap\ch(0,b)=\emptyset$, then $\ch(0,a)\subset\ch(l,c)$.
Since $|\ch(0,a)|=n$ and $|\ch(l,c)|=N$, we have $n=N$ and $\ch(0,a)=\ch(l,c)$, which contradicts Lemma \ref{lem201712130947}.
Therefore, $\ch(0,a)\cap\ch(0,b)\neq\emptyset$.

(2) Set $\ch(0,a)=\{\lambda_1,\ldots,\lambda_n\}$ and $\ch(0,b)=\{\mu_1,\ldots,\mu_m\}$ so that $\lambda_1=\mu_1,\ldots,\lambda_N\allowbreak=\mu_N$.
Choose $\lambda$ so that $\lambda+\lambda_{N+1}+\cdots+\lambda_n=\frac{1}{12}(n-N+1)(l+n-N)$.
By Lemma \ref{lem201712130954}, there exists $c\in\mcalr_{2n-2N+2}^{n-N+1}$ such that $\ch(l,c)=\{\lambda,\lambda_{N+1},\ldots,\lambda_n\}$.

(2-1) The case $l=2m$.
$\lambda+\sum_{i=N+1}^m\mu_i=\lambda+\sum_{i=1}^m\mu_{i}-\sum_{i=1}^N\mu_{i}=\lambda+\sum_{i=1}^m\mu_{i}-\sum_{i=1}^N\lambda_{i}=\lambda+\sum_{i=1}^m\mu_{i}-\sum_{i=1}^n\lambda_{i}+\sum_{i=N+1}^n\lambda_{i}=\sum_{i=1}^m\mu_{i}-\sum_{i=1}^n\lambda_{i}+(\lambda+\sum_{i=N+1}^n\lambda_{i})=\frac{1}{12}m(m-1)-\frac{1}{12}n(n-1)+\frac{1}{12}(n-N+1)(2m+n-N)=\frac{1}{12}(m-N+1)(2n+m-N)$.
By Lemma \ref{lem201712130954}, $\ch(2n,d)=\{\lambda,\mu_{N+1},\ldots,\mu_m\}$ for some $d\in\mcalr_{2m-2N+2}^{m-N+1}$.
Since $cb,da\in\mcalr_{2m+2n-2N+2}^{m+n-N+1}$ and $\ch(0,cb)=\ch(0,da)$, we have $cb-da\in Z\cap\calr_{2m+2n-2N+2}$.
By the assumption, $2m+2n-2N+2\le10$.
Therefore, $ Z\cap\calr_{2m+2n-2N+2}=\{0\}$ and so $cb=da$.

(2-2) The case $l\ge 2m+4$.
By Lemma \ref{lem201712130954}, $\ch(2n,d)=\{\lambda,\mu_{N+1},\ldots,\mu_m\}$ for some $d\in\mcalr_{l-2N+2}^{m-N+1}$.
The rest of the proof is clear.
\end{proof}

The following corollary gives a condition under which a quasimonic MLDO maps the solution space of a monic MLDE to another solution space.

\begin{cor}
\label{cor201801191507}
Let $k\in\zz_{\ge 0}$, $l,n,N\in\zz_{>0}$, $a\in\mcalr^n_{2n}$ and $b\in\acalr_l$.
Then the following hold.
\begin{enumerate}
\item
If $N\le n$ and there exists $c\in\mcalr^N_{2N}$ such that $b[k]$ maps the kernel of $a[k]$ to the kernel of  $c[k+l]$, then $\ch(k,a)\cap\ch(k,b)\neq\emptyset$.
\item
If $\ch(k,a)\cap\ch(k,b)\neq\emptyset$, $N\le|\ch(k,a)\cap\ch(k,b)|$ and $l+2n-2N\le8$, then there exists $c\in\mcalr_{2n-2N+2}^{n-N+1}$ such that $b[k]$ maps the kernel of $a[k]$ to the kernel of  $c[k+l]$.
\end{enumerate}
\end{cor}

\begin{proof}
Since $b[k]\ker(a[k])\subset\ker(c[k+l])$ is equivalent to $\ker(a[k])\subset\ker((cb)[k])$, the results follow from Theorem \ref{thm201801191511}. 
\end{proof}

\begin{rem}
Let $a\in\mcalr^n_{2n}$ and $b\in\calr_l$.
By Lemma \ref{lem201804181457}, there exist homogeneous $a'\in\calr$ and $c\in\mcalr$ such that $a'a=cb$, therefore $b[k](\ker (a[k]))\subset \ker (c[k+l])$.
\end{rem}

\subsection{Example I}
\label{sct201801121359}

We introduce a family of third order monic MLDEs $\phi_pf=0$ such that the solutions are related to the Dedekind eta function (Subsection \ref{sct201801121359}) and the theta functions of the $\rmd_n$ lattices (Subsection \ref{subsection201804181403}).
In the end of Subsection \ref{subsection201804181403}, we apply Corollary \ref{cor201801191507} to the solution space of $\phi_pf=0$ (Example \ref{ex201807171044}).

\begin{lem}
\label{lemma_eta_functions}
\begin{enumerate}
\item
$\eta(2z)|_{1/2}T=e(\frac{1}{12})\eta(2z)$ and $\eta(2z)|_{1/2}S=\frac{1}{\sqrt{2}}e(-\frac{1}{8})\eta(\frac{z}{2})$.
\item
$\eta(\frac{z}{2})|_{1/2}T=\eta(\frac{z+1}{2})$ and $\eta(\frac{z}{2})|_{1/2}S=\sqrt{2}e(-\frac{1}{8})\eta(2z)$.
\item
$\eta(\frac{z+1}{2})|_{1/2}T=e(\frac{1}{24})\eta(\frac{z}{2})$ and $\eta(\frac{z+1}{2})|_{1/2}S=e(-\frac{1}{8})\eta(\frac{z+1}{2})$.
\item
$\eta(2z)\eta(\frac{z}{2})\eta(\frac{z+1}{2})=e(\frac{1}{48})\eta(z)^3$.
\item
$16\eta(2z)^8+\eta(\frac{z}{2})^8-e(-\frac{1}{6})\eta(\frac{z+1}{2})^8=0$.
\end{enumerate}
\end{lem}

\begin{proof}
(1), (2) and the first part of (3) can be proved by the relations $\eta(z+1)=e(\frac{1}{24})\eta(z)$ and $\eta(-\frac{1}{z})=\sqrt{-iz}\eta(z)$.
We prove the second part of (3).
Since $\eta(\frac{Sz+1}{2})=\eta(\frac{1}{2}-\frac{1}{2z})=\eta((\begin{smallmatrix}1&-1\\2&-1\end{smallmatrix})\frac{z+1}{2})=\chi((\begin{smallmatrix}1&-1\\2&-1\end{smallmatrix}))\sqrt{z}\eta(\frac{z+1}{2})=e(-\frac{1}{8})\sqrt{z}\eta(\frac{z+1}{2})$, we have $\eta(\frac{z+1}{2})|_{1/2}S=\frac{1}{\sqrt{z}}\eta(\frac{Sz+1}{2})=e(-\frac{1}{8})\eta(\frac{z+1}{2})$.
In order to calculate $\chi((\begin{smallmatrix}1&-1\\2&-1\end{smallmatrix}))$, we have used Theorem 2 in Chapter 4 in \cite{knopp1970}.

(4) The equality is proved by the definition $\eta(z)=q^{1/24}(1-q)(1-q^2)\cdots$.

(5) By (1) through (3), the left-hand side is a modular form of weight $4$ on $\sltwoz$.
The identity follows from the valence formula (cf.\ Theorem 4.1.4 in \cite{rankin1977}).
\end{proof}

Let $p\in\rr$.
Since $\eta(z)$ never vanishes on $\calh$, $\eta^p(z)$ is defined on $\calh$ and has the $q$-expansion
$
q^{p/24}(1-pq+\frac{p(p-3)}{2}q^2+\cdots).
$
The $\cc$-vector space spanned by
$\eta^p(2z)=q^{p/12}(1-pq^2+\frac{p(p-3)}{2}q^4+\cdots)$, $\eta^p(\frac{z}{2})=q^{p/48}(1-pq^{1/2}+\frac{p(p-3)}{2}q+\cdots)$ and $\eta^p(\frac{z+1}{2})=e(\frac{p}{48})q^{p/48}(1+pq^{1/2}+\frac{p(p-3)}{2}q+\cdots)$
is invariant under the weight $\frac{p}{2}$ modular transformations.
We have the following $q$-expansions
\begin{align}
\eta^p\left(\frac{z}{2}\right)+e\left(-\frac{p}{48}\right)\eta^p\left(\frac{z+1}{2}\right)&=2q^{p/48}\left(1+\frac{p(p-3)}{2}q+\cdots\right),\\
-\eta^p\left(\frac{z}{2}\right)+e\left(-\frac{p}{48}\right)\eta^p\left(\frac{z+1}{2}\right)&=2q^{p/48}\left(pq^{1/2}+\cdots\right).
\end{align}
If $p\neq 0,8$, then the leading exponents of $\eta^p(2z)$, $\pm\eta^p(\frac{z}{2})+e(-\frac{p}{48})\eta^p(\frac{z+1}{2})$ are distinct real numbers.
By Mason's theorem, they satisfy a monic MLDE of weight $\frac{p}{2}$ and order $3$.
A monic MLDE of weight $\frac{p}{2}$ and order $3$ has the form $(\srdiff{3}{p/2}+x E_4 D_{p/2}+y E_6)f=0$ with $x,y\in\cc$.
Calculating the indicial roots, we can uniquely determine $x$ and $y$ as $x=-\frac{3p^2-24p+128}{2304}$ and $y=-\frac{p^2(p-24)}{55296}$.

\begin{prop}
\label{prop201801151023}
Let $p\in\rr$ and $\phi_p=\srdiff{3}{p/2}-\frac{3p^2-24p+128}{2304}E_4 D_{p/2}-\frac{p^2(p-24)}{55296}E_6$.
Then the following hold.
\begin{enumerate}
\item
If $p\neq 0,8$, then $\eta^p(2z)$, $\eta^p(\frac{z}{2})$ and $\eta^p(\frac{z+1}{2})$ are independent solutions to $\phi_p f=0$.
\item
If $p=0$, then $1$, $\int E_2^{(2)}(z)dz$ and $\int \sqrt{\Delta_4^{(2)}(z)}dz$ are independent solutions, where $E_2^{(2)}(z)=2E_2(2z)-E_2(z)$ and $\Delta_4^{(2)}(z)=\frac{\eta(2z)^{16}}{\eta(z)^8}$.
\item
If $p=8$, then $\eta^8(2z)$ and $\eta^8(\frac{z}{2})$ are independent solutions.
\end{enumerate}
\end{prop}

\begin{proof}
(1) Already proved.

(2) The linear independence is clear from the $q$-expansions.
According to Theorem 1 in \cite{kaneko_koike2003}, the solutions to $(\srdiff{2}{2}-\frac{1}{18}E_4)f=0$ are $E_2^{(2)}$ and $\sqrt{\Delta_4^{(2)}}$.

(3) By Lemma \ref{lemma_eta_functions} (5), the $\cc$-vector space spanned by
$\eta^8(2z)=q^{2/3}(1-8q^2+\cdots)$ and $\eta^8(\frac{z}{2})=q^{1/6}(1-8q^{1/2}+\cdots)$
is invariant under the weight $4$ modular transformations.
By Mason's theorem, they satisfy $(\srdiff{2}{4}-\frac{1}{18}E_4)f=0$.
\end{proof}

\begin{rem}
If $p\neq0,3,6$, then $\eta^p(3z)$, $\eta^p(\frac{z}{3})$, $\eta^p(\frac{z+1}{3})$ and $\eta^p(\frac{z+2}{3})$ are independent solutions to $\srdiff{4}{p/2}-\frac{p^2-6p+18}{216}E_4\srdiff{2}{p/2}-\frac{p^3-12p^2+45p-81}{5832}E_6D_{p/2}\allowbreak-\frac{p^2(p^2-36p+288)}{559872}{E_4}^2)f=0.$
\end{rem}

\subsection{Example II}
\label{subsection201804181403}

For an even lattice $L\subset\rr^n$ and a point $p\in\rr^n$, the theta function $\theta_{p+L}$ is given by $\theta_{p+L}(z)=\sum_{x\in L}q^{(x+p,x+p)/2},$ where $(\bullet,\bullet)$ is the standard inner product on $\rr^n$.
We denote by $L^*$ the dual lattice of $L$.
Then the $\cc$-vector space spanned by $\theta_{p+L}$ with $p\in L^*/L$ is invariant under the weight $\frac{n}{2}$ modular transformations (cf.\ Proposition 3.2 in \cite{ebeling2013}).

For $n\ge 3$, $\rmd_n$ are the root lattices $\{(x_1,\ldots,x_n)\in\zz^n\mid x_1+\cdots+x_n\in 2\zz\}$.
We have
$\rmd_n^*=\zz^n+\zz t$ and $\rmd_n^*/\rmd_n=\{0,s,t,s+t\}$,
where $s=(1,0,\ldots,0)$ and $t=(\frac{1}{2},\ldots,\frac{1}{2})$.

We formally set $\rmd_2=\sqrt{2}\zz\oplus\sqrt{2}\zz$ and $\rmd_1=2\zz$.

\begin{prop}
\label{prop201801151033}
For $n\ge 1$ and $p\in \rmd_n^*/\rmd_n$, $\theta_{p+\rmd_n}$ satisfies
\begin{equation}
\left(\srdiff{3}{n/2}-\frac{3n^2-12n+32}{576}E_4D_{n/2}-\frac{n^2(n-12)}{6912}E_6\right)f=0.
\end{equation}
Equivalently, it satisfies
$
\phi_{2n}(\eta^n\theta_{p+\rmd_n})=0.
$
\end{prop}

\begin{proof}
(1) The case $n\ge3$.
Since $\theta_2(z)=2\frac{\eta(2z)^2}{\eta(z)}$, $\theta_3(z)=e(-\frac{1}{24})\frac{\eta((z+1)/2)^2}{\eta(z)}$ and $\theta_4(z)=\frac{\eta(z/2)^2}{\eta(z)}$, the result follows from the formulae $\theta_{\rmd_n}=\frac{1}{2}({\theta_3}^n+{\theta_4}^n)$, $\theta_{s+\rmd_n}=\frac{1}{2}({\theta_3}^n-{\theta_4}^n)$ and $\theta_{t+\rmd_n}=\theta_{s+t+\rmd_n}=\frac{1}{2}{\theta_2}^n$ (cf.\ Chapter 4, Section 7.1 in \cite{conway_sloane1999}).
However, we give an alternative proof without using the formulae.
We have
\begin{align}
\theta_{s+t+\rmd_n}&=\sum_{x_1+\cdots+x_n\in2\zz}q^{((x_1+3/2)^2+(x_2+1/2)^2+\cdots+(x_n+1/2)^2)/2}\\
&=\sum_{x_1+\cdots+x_n\in2\zz}q^{((-x_1-1/2)^2+(x_2+1/2)^2+\cdots+(x_n+1/2)^2)/2}=\theta_{t+\rmd_n},
\end{align}
so the $\cc$-vector space spanned by $\theta_{\rmd_n}$, $\theta_{s+\rmd_n}$ and $\theta_{t+\rmd_n}$ is invariant under the weight $\frac{n}{2}$ modular transformations.
The leading terms are $\theta_{\rmd_n}=1+\cdots$, $\theta_{s+\rmd_n}=2nq^{1/2}+\cdots$ and $\theta_{s+t+\rmd_n}=2^{n-1}q^{n/8}+\cdots$.

(1-i) The case $n\ge3$ and $n\neq4$.
The result follows from Mason's theorem.

(1-ii) The case $n=4$.
By Theorem 3.2 in \cite{ebeling2013}, we have $\theta_{s+\rmd_n},\theta_{t+\rmd_n}\in\calm(2,\Gamma(2),1,\triv)$ and the valence formula shows $\theta_{s+\rmd_n}=\theta_{t+\rmd_n}$.
Therefore, $\theta_{\rmd_n}$ and $\theta_{s+\rmd_n}$ satisfy $(\srdiff{2}{2}-\frac{1}{18}E_4)f=0$.

(2) The case $n=1$.
The proof is the same as (1).

(3) The case $n=2$.
$\rmd_2^*=\frac{\sqrt{2}}{2}\zz\oplus\frac{\sqrt{2}}{2}\zz$ and $\rmd_2^*/\rmd_2=\{0,s',t',s'+t'\}$, where $s'=(\frac{\sqrt{2}}{2},0)$ and $t'=(0,\frac{\sqrt{2}}{2})$.
We have $\theta_{\rmd_n}=1+4q+4q^2+\cdots$, $\theta_{s'+\rmd_n}=\theta_{t'+\rmd_n}=2q^{1/4}+4q^{5/4}+2q^{9/4}+\cdots$ and $\theta_{s'+t'+\rmd_n}=4q^{1/2}+8q^{5/2}+4q^{9/2}+\cdots$.
\end{proof}

\begin{exmp}
\label{ex201807171044}
We apply Corollary \ref{cor201801191507} to the setting $k=0$, $n=3$, $N=1$, $m=2$, $l=4$ and $a={\dd}^3-\frac{3p^2-24p+128}{2304}e_4 \dd-\frac{p^2(p-24)}{55296}e_6\in\mcalr_6^3$.
It is clear that $\ch(0,a)=\{\frac{p}{24},-\frac{p}{48},\frac{1}{2}-\frac{p}{48}\}$.
The monic MLDO $b\in\acalr_4^2=\mcalr_4^2$ has the form $b={\dd}^2+x e_4$, where $x\in\cc$.

If $\frac{p}{24}\in\ch(0,b)$, then $x=-\frac{p(p-4)}{576}$ and $b[0]\ker(a[0])\subset\ker(c[4])$, where $c={\dd}^3-\frac{3p^2+72p+512}{2304}e_4\dd-\frac{(p+16)^2 (p-8)}{55296}e_6\in\mcalr_3^6$.

If $-\frac{p}{48}\in\ch(0,b)$, then $x=-\frac{p(p+8)}{2304}$ and $b[0]\ker(a[0])\subset\ker(c[4])$, where $c={\dd}^3-\frac{3p^2-72p+512}{2304}e_4\dd-\frac{(p-8)^2 (p-32)}{55296}e_6\in\mcalr_3^6$.

If $\frac{1}{2}-\frac{p}{48}\in\ch(0,b)$, then $x=-\frac{(p-16)(p-24)}{2304}$ and $b[0]\ker(a[0])\subset\ker(c[4])$, where $c={\dd}^3-\frac{3p^2-72p+1664}{2304}e_4\dd-\frac{(p+16)(p-8)(p-56)}{55296}e_6\in\mcalr_3^6$.
\end{exmp}

\section{Eisenstein series and monic MLDE}
\label{section201712301302}
Utilizing the properties of $\calr$ established in the preceding sections, we give a lower bound of the order of monic MLDEs of weight $4m+6n$ satisfied by ${E_4}^m{E_6}^n$.

By Theorem 1 in \cite{kaneko_koike2003}, the Eisenstein series $E_k$ for $k\in\{4,6,10\}$ satisfies the second order monic MLDE $(\srdiff{2}{k}-\frac{k(k+2)}{144}E_4)f=0$, the \textit{Kaneko-Zagier equation}.
Explicitly,
$(\srdiff{2}{4}-\frac{1}{6}E_4)E_4=0$, $(\srdiff{2}{6}-\frac{1}{3}E_4)E_6=0$ and $(\srdiff{2}{10}-\frac{5}{6}E_4)E_{10}=0$.
For $k\in\{4,6,10\}$, the independent solutions to $(\srdiff{2}{4}-\frac{k(k+2)}{144}E_4)f=0$ are $E_k$ and $F_k$, where $F_k$ is a $q$-series whose leading exponent is $\frac{(k+1)}{6}$.

\begin{prop}
\label{prop201801111416}
Let $n\in\zz_{\ge0}$ and $k\in\{4,6,10\}$.
Then ${E_k}^n$ satisfies a monic MLDE of weight $nk$ and order $n+1$.
\end{prop}

\begin{proof}
The case $n=0$ is trivial.
Assume $n>0$.
Since the $\cc$-vector space spanned by $E_k$ and $F_k$ is invariant under the weight $k$ modular transformations, the $\cc$-vector space spanned by ${E_k}^i {F_k}^{n-i}$ ($i=0,\ldots,n$) is invariant under the weight $nk$ modular transformations.
The leading exponents of ${E_k}^i {F_k}^{n-i}$ are $\frac{(n-i)(k+1)}{6}$, so the result follows from Mason's theorem.
\end{proof}

By Mason's theorem, ${E_4}^m{E_6}^n$ satisfies a (possibly non-monic) MLDE of weight $4m+6n$ and order $1$.
By Theorem \ref{thm201804191022} (4), it satisfies a monic MLDE of weight $4m+6n$ and some order.
The following theorem gives a lower bound for such orders.
Note that $E_{10}=E_4E_6$.

\begin{thm}
\label{thm201801111417}
Let $m,n\in\zz_{\ge0}$ and $N\in\zz_{>0}$.
Suppose that ${E_4}^m{E_6}^n$ satisfies a monic MLDE of weight $4m+6n$ and order $N$.
Then $N\ge\max\{m,n\}+1$.
\end{thm}

\begin{proof}
Suppose $(\srdiff{N}{4m+6n}+g_2 \srdiff{N-2}{4m+6n}+\cdots+g_N){E_4}^m{E_6}^n=0$ with $g_i\in\calm_{2i}$.
Set $a={\dd}^N+g_2 {\dd}^{N-2}+\cdots+g_N\in\mcalr_{2N}^N$ and $b=e_4 e_6 \dd+\frac{n}{2}{e_4}^3+\frac{m}{3}{e_6}^2\in\acalr_{12}^1$.
Then $b({E_4}^m{E_6}^n)=0$.
By the division of MLDO, ${e_4}^N{e_6}^N a=cb+c'$ for some $c\in\calr_{12N-12}^{N-1}$ and $c'\in\calr_{12N}^{<1}$.
Since ${E_4}^m{E_6}^n\neq0$, we have $c'=0$.
Set $c=h_1 {\dd}^{N-1}+\cdots +h_N$, where $h_i\in\calm_{10N+2i-12}$.
Comparing the tops of ${e_4}^N{e_6}^N a=cb$, we see $h_1={e_4}^{N-1}{e_6}^{N-1}$.
Comparing the coefficients of ${\dd}^{N-i}$ for $i\in\{1,\ldots,N-1\}$, we see
\begin{align}
\label{eq201712161731}
{e_4}^N{e_6}^N g_i=&e_4e_6 h_{i+1}+\sum_{j=1}^i h_j\binom{N-j}{N-i-1}D^{i-j+1}[e_4e_6]\nonumber\\
&+\sum_{j=1}^i h_j\binom{N-j}{N-i}D^{i-j}\left[\frac{n}{2}{e_4}^3+\frac{m}{3}{e_6}^2\right],
\end{align}
since the coefficient of ${\dd}^{N-i}$ in $h_j{\dd}^{N-j}(e_4 e_6\dd+\frac{n}{2}{e_4}^3+\frac{m}{3}{e_6}^2)$ is
\begin{equation}
\begin{cases}
h_j(\binom{N-j}{N-i-1}D^{i-j+1}[e_4e_6]+\binom{N-j}{N-i}D^{i-j}[\frac{n}{2}{e_4}^3+\frac{m}{3}{e_6}^2]) & \text{if }1\le j\le i,\\
e_4e_6 h_j & \text{if }j=i+1,\\
0 & \text{if } i+1<j\le N.
\end{cases}
\end{equation}

Let us prove
\begin{align}
\label{201712221059}
h_i=&{e_4}^{N-i}{e_6}^{N-i}\left( \frac{(N-(n+i-1))\cdots(N-(n+1))}{2^{i-1}}{e_4}^{3i-3}\right.\nonumber\\
&\left.+\frac{(N-(m+i-1))\cdots(N-(m+1))}{3^{i-1}}{e_6}^{2i-2}+e_4 e_6 R_i \right)
\end{align}
by induction on $i\ge2$, where $R_i\in\qq[e_4,e_6,g_1,\ldots,g_{i-1}]\subset\calm$.

(1) By setting $i=1$ in Eq.\ (\ref{eq201712161731}),
\begin{equation}
{e_4}^N{e_6}^N g_1=e_4e_6h_2+h_1\binom{N-1}{N-2}D[e_4e_6]+h_1\left(\frac{n}{2}{e_4}^3+\frac{m}{3}{e_6}^2\right),
\end{equation}
so $h_2={e_4}^{N-2}{e_6}^{N-2}(\frac{N-(n+1)}{2}{e_4}^3+\frac{N-(m+1)}{3}{e_6}^2+e_4e_6g_1)$.

(2) Assume that the result holds for all $i'\le i$.
By Eq.\ (\ref{eq201712161731}),
\begin{align}
&e_4e_6h_{i+1}={(e_4 e_6)}^N g_i-h_i\binom{N-i}{N-i-1}D[e_4e_6]-h_i\left(\frac{n}{2}{e_4}^3+\frac{m}{3}{e_6}^2\right)\nonumber\\
&\quad-\sum_{j=1}^{i-1} h_j\binom{N-j}{N-i-1}D^{i-j+1}[e_4e_6]-\sum_{j=1}^{i-1} h_j\binom{N-j}{N-i}D^{i-j}\left[\frac{n}{2}{e_4}^3+\frac{m}{3}{e_6}^2\right]\\
&=h_i\left(\frac{N-(n+i)}{2}{e_4}^3+\frac{N-(m+i)}{3}{e_6}^2\right)+{(e_4 e_6)}^{N-i+1}\qq[e_4,e_6,g_1,\ldots,g_{i-2},g_i]\\
&={(e_4 e_6)}^{N-i}\frac{(N-(n+i))\cdots(N-(n+1))}{2^i}{e_4}^{3i}\nonumber\\
&\quad+{(e_4 e_6)}^{N-i}\frac{(N-(m+i))\cdots(N-(m+1))}{3^i}{e_6}^{2i}+{(e_4 e_6)}^{N-i+1}\qq[e_4,e_6,g_1,\ldots,g_i],
\end{align}
therefore Eq.\ (\ref{201712221059}) holds for $i+1$.

By comparing the coefficients of ${\dd}^0$ in ${e_4}^N{e_6}^N a=cb$, it follows from Eq.\ (\ref{201712221059}) that
\begin{align}
{e_4}^N{e_6}^Ng_N&=\sum_{i=0}^Nh_iD^{N-i}\left[\frac{n}{2}{e_4}^3+\frac{m}{3}{e_6}^2\right]\\
&=\frac{n(N-(n+N-1))\cdots(N-(n+1))}{2^N}{e_4}^{3N}\nonumber\\
&\quad+\frac{m(N-(m+N-1))\cdots(N-(m+1))}{3^N}{e_6}^{2N}+e_4e_6\calm.
\end{align}
Therefore, $n(N-(n+N-1))\cdots(N-(n+1))=0$ and $m(N-(m+N-1))\cdots(N-(m+1))=0$, so that $N>n$ and $N>m$.
\end{proof}

For $\phi\in H=\bigoplus_{n\in\rr}H_n$, we denote by $\mord(\phi)$ the least nonnegative integer $n$ such that some element of $\mcalr^n$ annihilates $\phi$.
If there is no such $n$, we formally set $\mord(\phi)=\infty$.
We call $\mord(\phi)$ the \textit{modular order} of $\phi$.

By Proposition \ref{prop201801111416} and Theorem \ref{thm201801111417}, $\mord({E_4}^m{E_6}^n)\ge\max\{m,n\}+1$ and $\mord({E_4}^n)=\mord({E_6}^n)=\mord({E_4}^n{E_6}^n)=n+1$.
We conjecture the following.
\begin{conj}
$\mord({E_4}^m{E_6}^n)=\max\{m,n\}+1$ for all $m,n\ge0$.
\end{conj}

\section{Miscellaneous results}
\label{section20180220}
In this section, we show some properties of $\calr$ which are not relevant to the preceding sections.

\subsection{Ideals of $\calr$}
In this subsection, we study some ideals of $\calr$.

\begin{lem}
The center $Z(\calr)$ is equal to $\langle \Delta\rangle$ (the $\cc$-subalgebra of $\calr$ generated by $\Delta$).
\end{lem}

\begin{proof}
Since $[\dd,1]=[\dd,\Delta]=0$, $Z(\calr)\supset \langle \Delta\rangle$.
To prove $Z(\calr)\subset\langle \Delta\rangle$, let $a\in Z(\calr)$ be homogeneous.
Set $a=\sum_{i=0}^n a_i {\dd}^i$, where $a_i\in\calm(l-2i)$ for some $l\in\zz$ and $a_n\neq 0$. 
By comparing the coefficients of ${\dd}^{n-1}$ in $a e_4=e_4 a$, we have $a_{n-1}e_4+n a_nD[e_4]=e_4 a_{n-1}$, so $n=0$ and $a=a_0\in\calm(l)$.
Since $\dd a=a\dd$, we have $0=D[a]=D_l a$, so that the leading coefficient of $a$ is $\frac{l}{12}$.
Since the leading coefficient of an element of $\calm(l)$ is a nonnegative integer, $\frac{a}{\Delta^{l/12}}\in\calm(0)=\cc$, which means $a\in\langle \Delta\rangle$.
\end{proof}

\begin{prop}
$\Delta\calr$ is a completely prime two-sided ideal of $\calr$. 
\end{prop}

\begin{proof}
Since $\Delta\in Z(\calr)$, $\Delta\calr$ is a two-sided ideal.
Note that $\Delta\calr=\{\sum_{i=0}^n a_i {\dd}^i\mid a_i\in\cals,\ n\ge 0\}$.
Let $a,b\in\calr\backslash\{0\}$ such that $ab\in\Delta\calr$ and $a\notin\Delta\calr$.
Set $a=\sum_{i=0}^n a_i {\dd}^i$ and $b=\sum_{j=0}^m b_j {\dd}^j$, where $a_i,b_j\in\calm$, $a_n,b_m\neq0$ and $n,m\ge 0$.
There exists the maximum number $N=\max\{i\in[0,n]\mid a_i\notin\cals\}$ and $\sum_{i=0}^N a_i {\dd}^i \sum_{j=0}^m b_j {\dd}^j\in\Delta\calr$.
Since the top is $a_N b_m$, we have $a_N b_m\in\cals$ and $b_m\in\cals$, therefore $\sum_{i=0}^N a_i {\dd}^i \sum_{j=0}^{m-1}b_j {\dd}^j\in\Delta\calr$.
By induction on $j$, we see that every $b_j$ belongs to $\cals$ and $b\in\Delta\calr$.
\end{proof}

Since $\calr$ admits the division with remainder, it has a PID-like property.
For $a\in\calr$, let $[a]$ denote $\{b\in\calr \mid \text{there exists }f\in\calm\backslash\{0\}\text{ such that }fb\in\calr a\}$.

\begin{thm}
\label{thm201804181551}
\begin{enumerate}
\item
For any $a\in\calr$, $[a]$ is a left ideal of $\calr$.
\item
For any left ideal $I$ of $\calr$, there exists $a\in I$ such that $\calr a\subset I \subset [a]$.
\item
For any $\phi\in H$, there exists $a\in\calr$ such that $\ann_{\calr}(\phi)=[a]$.
\end{enumerate}
\end{thm}

\begin{proof}
(1) It is clear that $[a]$ is an abelian group. 
Let $b\in[a]$ and $c\in\calr$.
Then $fb=pa$ for some $f\in\calm\backslash\{0\}$ and $p\in\calr$.
By the division, $gc=qf$ for some $g\in\calm\backslash\{0\}$ and $q\in\calr$.
Therefore, $g(cb)=qfb=qpa\in\calr a$ and $cb\in[a]$.

(2) If $I \neq\{0\}$, choose an element $a\in I$ of the lowest order.

(3) If $\ann_{\calr}(\phi)\neq\{0\}$, choose an element $a\in \ann_{\calr}(\phi)$ of the lowest order.
\end{proof}

\begin{rem}
Just like the division properties, Theorem \ref{thm201804181551} holds for a skew polynomial ring $S=R[x;1,d]$, where $R$ is a commutative integral domain and $d$ is a derivation of $R$, and for a left $S$-module $M$ such that $fm=0$ implies $m=0$ for $f\in R\backslash\{0\}$ and $m\in M$.
\end{rem}

\subsection{$\cc$-algebras containing $\calr$ and their endomorphisms}

\begin{lem}
\label{201712241805}
Suppose
\begin{equation}
\sum_{i,j,k,l\ge0,\ 2i+2j+4k+6l=n}C_{ijkl}z^{-i}{E_2}^j{E_4}^k{E_6}^l=0, \label{eq201802231002}
\end{equation}
where $C_{ijkl}\in\cc$ and $n\ge0$.
Then $C_{ijkl}=0$ for all $i,j,k,l\ge 0$ with $2i+2j+4k+6l=n$.
\end{lem}

\begin{proof}
For $p\in\zz$, set $A_p=(\begin{smallmatrix}1&p\\1&p+1\end{smallmatrix})\in\sltwoz$.
By replacing $z$ with $A_p z$ and dividing the both sides of Eq.\ (\ref{eq201802231002}) by $(z+p+1)^n$, we have
\begin{equation}
\sum_{2i+2j+4k+6l=n}C_{ijkl}(z+p)^{-i}(z+p+1)^{-i}\left(E_2+\frac{6}{\pi i(z+p+1)}\right)^j{E_4}^k{E_6}^l=0.
\end{equation}
Letting $p\to\infty$, we obtain $\sum_{2j+4k+6l=n}C_{0jkl}{E_2}^j{E_4}^k{E_6}^l=0$ and $C_{0jkl}=0$ for all $j,k,l\ge0$ with $2j+4k+6l=n$.
The result is proved by induction on $i$.
\end{proof}

Consider the following three homogeneous elements of $\End(H)$:
\begin{align}
e_1&\in \End^1(H),\quad (f,k)\mapsto (f,k+1),\\
e_2&\in \End^2(H),\quad (f,k)\mapsto (E_2 f,k+2),\\
f_2&\in \End^2(H),\quad (f,k)\mapsto \left(\frac{f}{\log q},k+2\right).
\end{align}
It is straightforward to show $[\dd,e_1]=-\frac{1}{12}e_1 e_2$, $[\dd,e_2]=-\frac{1}{12}({e_2}^2+e_4)$ and $[\dd,f_2]=-{f_2}^2-\frac{1}{6}f_2 e_2$.
Consider the following three graded $\cc$-subalgebra of $\End(H)$:
\begin{align}
\calqr&=\langle e_2,e_4,e_6,\dd \rangle,\\
\calqreone&=\langle e_1,e_2,e_4,e_6,\dd \rangle,\\
\calqrftwo&=\langle f_2,e_2,e_4,e_6,\dd \rangle.
\end{align}

\begin{thm}
\begin{enumerate}
\item
The set $\{{e_2}^j{e_4}^k{e_6}^l{\dd}^m\mid j,k,l,m\ge0\}$ forms a $\cc$-basis of $\calqr$.
\item
The set $\{{e_1}^i{e_2}^j{e_4}^k{e_6}^l{\dd}^m\mid i,j,k,l,m\ge0\}$ forms a $\cc$-basis of $\calqreone$.
\item
The set $\{{f_2}^i{e_2}^j{e_4}^k{e_6}^l{\dd}^m\mid i,j,k,l,m\ge0\}$ forms a $\cc$-basis of $\calqrftwo$.
\end{enumerate}
\end{thm}

\begin{proof}
The proof is similar to that of Theorem \ref{thm201712301243}.
Utilize Lemma \ref{201712241805} for (3).
\end{proof}

\begin{thm}
\label{thm201806010922}
\begin{enumerate}
\item
As graded $\cc$-algebras, $\calqr$ is isomorphic to $T(E_2\oplus E_4\oplus E_6\oplus \dd)/I_1$, where $I_1$ is the two-sided ideal generated by $[\dd,E_2]+\frac{1}{12}({E_2}\otimes E_2+E_4)$, $[\dd,E_4]+\frac{1}{3}E_6$, $[\dd,E_6]+\frac{1}{2}{E_4}\otimes E_4$, $[E_2,E_4]$, $[E_2,E_6]$ and $[E_4,E_6]$.
\item
As graded $\cc$-algebras, $\calqreone$ is isomorphic to $T(E_1\oplus E_2\oplus E_4\oplus E_6\oplus \dd)/I_2$, where $I_2$ is generated by $[\dd,E_1]+\frac{1}{12}E_1\otimes E_2$, $[\dd,E_2]+\frac{1}{12}({E_2}\otimes E_2+E_4)$, $[\dd,E_4]+\frac{1}{3}E_6$, $[\dd,E_6]+\frac{1}{2}{E_4}\otimes E_4$, $[E_1,E_2]$, $[E_1,E_4]$, $[E_1,E_6]$, $[E_2,E_4]$, $[E_2,E_6]$ and $[E_4,E_6]$.
\item
As graded $\cc$-algebras, $\calqrftwo$ is isomorphic to $T(F_2\oplus E_2\oplus E_4\oplus E_6\oplus \dd)/I_3$, where $I_3$ is generated by $[\dd,F_2]+{F_2}\otimes F_2+\frac{1}{6}F_2\otimes E_2$, $[\dd,E_2]+\frac{1}{12}({E_2}\otimes E_2+E_4)$, $[\dd,E_4]+\frac{1}{3}E_6$, $[\dd,E_6]+\frac{1}{2}{E_4}\otimes E_4$, $[F_2,E_2]$, $[F_2,E_4]$, $[F_2,E_6]$, $[E_2,E_4]$, $[E_2,E_6]$ and $[E_4,E_6]$.
\end{enumerate}
\end{thm}

\begin{proof}
The proof is similar to that of Theorem \ref{thm201801111434} (1).
\end{proof}

We set $E_1=(1,1)\in H_1$ and $F_2=(\frac{1}{\log q},2)\in H_2$.
It is easy to see that $\{E_1,E_2,E_4,E_6\}$ and $\{F_2,E_2,E_4,E_6\}$ are algebraically independent in $H$ (cf. Lemma \ref{201712241805}).
Note $\dd E_1=-\frac{1}{12}E_1 E_2$, $\dd F_2=-{F_2}^2-\frac{1}{6}F_2 F_2$ and $\dd E_2=-\frac{1}{12}({E_2}^2+E_4)$.

\begin{thm}
\begin{enumerate}
\item
As $\cc$-algebras, $\calqr$ is isomorphic to $\calqm[\xi;1,\dd]$.
\item
As $\cc$-algebras, $\calqreone$ is isomorphic to $\cc[E_1,E_2,E_4,E_6][\xi;1,\dd]$.
\item
As $\cc$-algebras,  $\calqrftwo$ is isomorphic to $\cc[F_2,E_2,E_4,E_6] [\xi;1,\dd]$.
\end{enumerate}
\end{thm}

\begin{proof}
The proof is similar to that of Theorem \ref{thm201805311336}.
\end{proof}

The grade-preserving endomorphisms of $\calr$, $\calqr$, $\calqreone$ and $\calqrftwo$ can be directly calculated by Theorem \ref{thm201806010922}.
Let $\mathcal{X}$ be either $\calr$, $\calqr$, $\calqreone$ or $\calqrftwo$.
Then the monoid of grade-preserving endomorphisms of $\mathcal{X}$ consists of the following:
For $\mathcal{X}=\calr$,
\begin{align}
(e_4,e_6,\dd)&\mapsto(a^2e_4,a^3e_6,a\dd).
\end{align}
For $\mathcal{X}=\calqr$, 
\begin{align}
(e_2,e_4,e_6,\dd)&\mapsto(ae_2,a^2e_4,a^3e_6,a\dd+be_2),\\
&\mapsto(0,0,0,a\dd+be_2).
\end{align}
For $\mathcal{X}=\calqreone$, 
\begin{align}
(e_1,e_2,e_4,e_6,\dd)&\mapsto(ae_1,be_2,b^2e_4,b^3e_6,b\dd+c{e_1}^2+de_2),\\
&\mapsto(0,0,0,0,b\dd+c{e_1}^2+de_2).
\end{align}
For $\mathcal{X}=\calqrftwo$, 
\begin{align}
(f_2,e_2,e_4,e_6,\dd)&\mapsto(af_2,ae_2,a^2e_4,a^3e_6,a\dd+bf_2+ce_2),\\
&\mapsto(-af_2,12af_2+ae_2,a^2e_4,a^3e_6,a\dd+bf_2+ce_2), \label{eq201802221820} \\
&\mapsto(0,12af_2+ae_2,a^2e_4,a^3e_6,a\dd+bf_2+ce_2),\\
&\mapsto(0,ae_2,a^2e_4,a^3e_6,a\dd+bf_2+ce_2),\\
&\mapsto(0,0,0,0,a\dd+bf_2+ce_2),
\end{align}
where $a,b,c,d\in\cc$ are arbitrary.

As an application of the endomorphisms above, we prove Proposition \ref{prop201805061103}.
Note that $(\dd-a {e_1}^2)(q^a,0)=0$ and $(\dd-a f_2)((\log q)^a,0)=0$ for $a\in\cc$.

\begin{prop}
\label{prop201805061103}
Let $F_i(x,y,z),G(x,y,z)\in\cc[x,y,z]$ be polynomials for $1\le i\le n$.
Suppose that $\wt (F_i(e_2,e_4,e_6))+2i=l$ is constant for $0\le i\le n$, and $\wt (G(e_2,e_4,e_6))=k$.
If $\sum_{i=0}^n F_i(E_2,E_4,E_6)\srdiff{i}{k} G(E_2,E_4,E_6)=0$, then
\begin{equation}
\sum_{i=0}^n F_i\left(\frac{12}{\log q}+E_2,E_4,E_6\right)\left(D_k-\frac{a}{\log q}\right)^{(i)} \left((\log q)^a G\left(\frac{12}{\log q}+E_2,E_4,E_6\right)\right)=0
\end{equation}
for $a\in\cc$, where $(D_k-\frac{a}{\log q})^{(i)}=(D_{k+2i-2}-\frac{a}{\log q})\cdots (D_{k+2}-\frac{a}{\log q})(D_k-\frac{a}{\log q})$.
\end{prop}

\begin{proof}
Since $\sum_{i=0}^n F_i(e_2,e_4,e_6){\dd}^i G(e_2,e_4,e_6)(1,0)=0$, $\sum_{i=0}^n F_i(e_2,e_4,e_6){\dd}^i G(e_2,e_4,e_6)\allowbreak=x\dd$ for some $x\in\calqr$.
By applying the endomorphism (\ref{eq201802221820}) with $a=1$ and $c=0$, we obtain $\sum_{i=0}^n F_i(12 f_2+e_2,e_4,e_6)(\dd+b f_2)^i G(12 f_2+e_2,e_4,e_6)=x'(\dd+b f_2)$ for some $x'\in\calqrftwo$.
Replace $b$ with $-a$.
Then $0=\sum_{i=0}^n F_i(12 f_2+e_2,e_4,e_6)(\dd-a f_2)^i G(12 f_2+e_2,e_4,e_6)((\log q)^a,0)=(\sum_{i=0}^n F_i(\frac{12}{\log q}+E_2,E_4,E_6)(D_k-\frac{a}{\log q})^{(i)} ((\log q)^a G(\frac{12}{\log q}+E_2,E_4,E_6)),\allowbreak k+l)$.
\end{proof}

\section*{Acknowledgement}
The author is very grateful to his advisor Dr. Atsushi Matsuo for his enlightening advice and the correction of the paper.
The author also thanks Leading Graduate Course for Frontiers of Mathematical Sciences and Physics, the University of Tokyo for financial support.

\end{document}